\documentclass[a4paper, 12pt]{article}
\usepackage{amsmath, amsthm, amssymb}
\usepackage{enumerate}
\usepackage{tikz}
\usepackage{fullpage}
\usepackage{enumitem}
\usepackage{hyperref,cleveref}
\hypersetup{hidelinks}
\newcommand\emaillink[1]{\href{mailto:#1}{\nolinkurl{#1}}}
\usepackage{mleftright}
\mleftright

\title{Shotgun assembly of random graphs}
\author{Tom Johnston\thanks{School of Mathematics, University of Bristol, Bristol, BS8\thinspace1UG, UK and Heilbronn Institute for Mathematical Research, Bristol, UK. Email: \emaillink{tom.johnston@bristol.ac.uk}.} \and Gal Kronenberg\thanks{Mathematical Institute, University of Oxford, Andrew Wiles Building, Radcliffe Observatory Quarter, Woodstock Road, Oxford, United Kingdom.} \thanks{Email: \emaillink{kronenberg@maths.ox.ac.uk}. Supported by the European Union’s Horizon
2020 research and innovation programme under the Marie Sk\l odowska Curie grant agreement No. 101030925. }\and Alexander Roberts \and Alex Scott\footnotemark[2] \thanks{Email: \emaillink{scott@maths.ox.ac.uk}. Research supported by EPSRC grant EP/X013642/1.}}

\newtheorem{theorem}{Theorem}
\newtheorem{lemma}[theorem]{Lemma}
\newtheorem{claim}[theorem]{Claim}

\theoremstyle{definition}
\newtheorem{remark}{Remark}
\newtheorem{fact}{Fact}
\newtheorem*{qu*}{Question}

\DeclareMathOperator{\E}{\mathbb{E}}
\DeclareMathOperator{\bP}{\mathbb{P}}
\DeclareMathOperator{\Var}{Var}
\DeclareMathOperator{\Bin}{Bin}

\newcommand{\expec}[1]{\E \left[#1\right]}
\newcommand{\prob}[1]{\bP\left(#1 \right)}

\newcommand{\floor}[1]{\left\lfloor #1 \right\rfloor}
\newcommand{\ceil}[1]{\left\lceil #1 \right\rceil}
\newcommand{\expb}[1]{\exp \left( #1 \right)}

\let \eps \varepsilon

\begin{document}

\maketitle

\begin{abstract}
    \noindent
  In the graph shotgun assembly problem, we are given the balls of radius $r$ around each vertex of a graph and asked to reconstruct the graph.
  We study the shotgun assembly of the Erd\H{o}s-R\'enyi random graph $\mathcal G(n,p)$ for a wide range of values of $r$.
  We determine the threshold for reconstructibility for each $r\geq 3$, extending and improving substantially on results of Mossel and Ross for $r=3$. For $r=2$, we give upper and lower bounds that improve on results of Gaudio and Mossel by polynomial factors. We also give a sharpening of a result of Huang and Tikhomirov for $r=1$.
\end{abstract}

\section{Introduction}

When can we reconstruct a graph from local information?
In the \emph{shotgun assembly problem}, we are given the balls $N_r(v)$ of radius $r$ around each vertex of a graph $G$ and aim to reconstruct the graph from this information.  Problems of this type arise naturally in DNA shotgun assembly, where the goal is to reconstruct a DNA sequence from a collection of shorter stretches of the sequence (see \cite{dyer1994probability,arratia1996poisson,motahari2013information} among many references), and have also been considered in the neural network context \cite{soudry2015efficient}. The shotgun assembly problem for random graphs was introduced in an influential paper of Mossel and Ross \cite{mossel2017shotgun}, which also raised a number of interesting variants such as the reconstruction of random jigsaws (see \cite{nenadov2016unique,martinsson2016shotgun,balister2017reconstructing,martinsson2019linear,bordenave2020shotgun}) and random colourings of lattices (see \cite{przykucki2022shotgun, ding2023shotgun}). There has also recently been work on the closely related problem of reconstructing random pictures \cite{narayanan2022reconstructing}.

In this paper we will be concerned with the shotgun assembly of an Erd\H os-R\'enyi random graph $G\in \mathcal G(n,p)$, where each edge is open independently with probability $p = p(n)$. This problem has already been extensively studied  \cite{mossel2017shotgun,gaudio2020shotgun,huang2021shotgun,ding2022shotgun} (there is also interesting work on other random graphs including random regular graphs \cite{mossel2015shotgun}, random geometric graphs \cite{adhikari2022geometric} and random simplicial complexes \cite{adhikari2022shotgun}).
Let us start by defining the problem more carefully.
For a graph $G$, let ${N}_r^{(G)}(v)$ be the graph induced by the vertices at distance at most $r$ from $v$, where the vertices are unlabelled except for the vertex $v$.
For an integer $r\ge1$ and graphs $G$ and $H$, we say $G$ and $H$ {\em have isomorphic $r$-neighbourhoods} if there is a bijection $\phi:V(G)\to V(H)$ such that for each vertex $v$ of $G$ there is an isomorphism from the $r$-neighbourhood $N^{(G)}_r(v)$ around $v$ in $G$ to the $r$-neighbourhood $N^{(H)}_r(\phi(v))$ around $\phi(v)$ in $H$ which maps $v$ to $\phi(v)$.  We say that $G$ is  \emph{reconstructible from its $r$-neighbourhoods} (or \textit{$r$-reconstructible}) if every graph with $r$-neighbourhoods isomorphic to those of $G$ is isomorphic to $G$.  The general problem is to determine for what range of $p$ a random graph $G\in \mathcal G(n,p)$ is reconstructible (or non-reconstructible) from its $r$-neighbourhoods with high probability (i.e.~with probability tending to 1 as $n$ tends to infinity).  We improve on previous bounds for all values of $r$, and give a fairly complete picture for $r\ge 3$.

For very small $p$, the general picture is similar for all $r$.  Indeed, we show that at every radius $r$ there is a phase transition when $p$ is around  $n^{-\frac{2r+1}{2r}}$.
If $p = o(n^{-\frac{2r+1}{2r}})$, there are no paths of lengths $2r$ with high probability and every component is contained entirely in an $r$-ball. This means we reconstruct the graph by iteratively identifying and removing the largest components. On the other hand, if $p$ grows slightly faster than $n^{-\frac{2r+1}{2r}}$, then with high probability we obtain a graph that is not $r$-reconstructible.

The more difficult question is what happens for larger $p$.  It seems likely that for every radius $r$ there should be a second phase transition around some threshold $t=t(n)$. By which we mean that, if $p=\omega(t(n))$, then $G$ is with high probability reconstructible from its $r$-neighbourhoods, while if $p=o(t(n))$ and $p=\omega(n^{-\frac{2r+1}{2r}})$, then with high probability $G$ is not reconstructible from its $r$-neighbourhoods.  This was not previously known at any radius.  Our results here prove the existence of this second phase transition for all $r\ge 3$, and narrow the gap for $r=1,2$.
We start by giving our main results regarding $r\geq 3$, and then we discuss reconstruction from the 1- and 2-neighbourhoods and give some small improvements.

\bigskip\noindent{\em Radius 3:}
We begin by looking at reconstruction from balls of radius 3, and give the correct threshold for when $\mathcal{G}(n,p)$ is 3-reconstructible with high probability.
Mossel and Ross \cite{mossel2017shotgun} considered reconstruction from balls of radius 3 and showed that $G\in\mathcal{G}(n,p)$ is with high probability 3-reconstructible when $p=\omega(\log^2(n)/n)$. We improve on this result, and show that there are two phase transitions: the first is around $n^{-7/6}$, and the second is around $\frac {\log^2n}{n(\log \log n)^3}$.

\begin{theorem}\label{thm:r=3}
    Let $p=p(n)$ and $G\in\mathcal G(n,p)$.
    There exist $\beta>\alpha>0$ such that the following hold.
    \begin{enumerate}[label=(\roman*)]
        \item If $p = o(n^{-7/6})$, then $G$ is reconstructible from its $3$-neighbourhoods with high probability.
        \item If ${p}=\omega({n^{-7/6}})$ and $p \le \alpha \tfrac{\log^2n}{n(\log\log n)^3}$, then with high probability $G$ is not reconstructible from its $3$-neighbourhoods.
        \item If  $p \ge \beta\tfrac{\log^2n}{n(\log\log n)^3}$, then $G$ is reconstructible from its $3$-neighbourhoods with high probability.
    \end{enumerate}
\end{theorem}

\bigskip\noindent{\em Radius 4 or more:}
A similar picture holds for any fixed radius $r\ge 4$ (and in fact even when $r$ grows slowly with $n$), except that the second phase transition comes earlier by a factor of roughly $\frac{\log n}{(\log \log n)^3}$.

\begin{theorem}\label{thm:r>3} Let $p=p(n)$ and $G\in\mathcal G(n,p)$.
    There exist $\beta>\alpha>0$ such that the following hold for all $4 \le r = o(\log n)$.
    \begin{enumerate}[label=(\roman*)]
        \item If $p = o(n^{-\frac{2r+1}{2r}})$, then $G$ is reconstructible from its $r$-neighbourhoods with high probability.
        \item If
              $p=\omega(n^{-\frac{2r+1}{2r}})$ and $p \le \alpha \tfrac{\log n}{rn}$, then with high probability $G$ is not reconstructible from its $r$-neighbourhoods.
        \item If $p \ge \beta\tfrac{\log n}{rn}$, then $G$ is reconstructible from its $r$-neighbourhoods with high probability.
    \end{enumerate}
\end{theorem}

Recently, Gaudio, R\'acz and Sridhar \cite{gaudio2022average} independently studied the special case of $r=4$ as part of their work on local canonical labellings of Erd\H{o}s-R\'enyi graphs and showed that $G \in \mathcal{G}(n,p)$ is 4-reconstructible with high probability when $np \geq (1 + \delta) \log n$. 

\bigskip\noindent{\em Radius 2:}  We next move to the case where $r=2$. It is not hard to see that if $p = \omega(\sqrt{{\log (n)}/{n}})$, then $G\in\mathcal G(n,p)$ is 2-reconstructible with high probability as the diameter of $G$ is at most 2 with high probability (and so the 2-balls are the entire graph). Better bounds were given by Gaudio and Mossel \cite{gaudio2020shotgun} who showed that, for any $\eps > 0$, $G$ is 2-reconstructible with high probability when $n^{-3/5 + \eps} \leq p \leq n^{-1/2 - \eps}$ or $p \geq n^{-1/2 + \eps}$.  We extend the range at the lower end, and remove the gap in the middle.

\begin{theorem}
    \label{thm:r=2}  Let $p=p(n)$ and $G\in\mathcal G(n,p)$.
    There exists a constant $\delta > 0$ such that the following holds. If $p \geq n^{-2/3 - \delta}$,  then $G$ is reconstructible from its 2-neighbourhoods with high probability.
\end{theorem}

For slightly sparser graphs, Gaudio and Mossel \cite{gaudio2020shotgun} showed that $G$ fails to be 2-reconstructible with high probability when $n^{-1 + \eps} \leq p \leq n^{-3/4 - \eps}$. We extend this range in both directions as follows.

\begin{theorem}\label{thm:r=2-non}  Let $p=p(n)$ and $G\in\mathcal G(n,p)$.
    If  $p \leq \frac{1}{3} n^{-3/4} \log^{1/4}n$ and $p=\omega(n^{-5/4})$, then with high probability $G$ cannot be reconstructed from its 2-neighbourhoods.
\end{theorem}

Once again,  the lower bound on $p$ in \Cref{thm:r=2-non} is best possible.

\begin{theorem}\label{thm:r=2-lower}  Let $p=p(n)$ and $G\in\mathcal G(n,p)$.
    If $p = o(n^{-5/4})$, then with high probability $G$ is reconstructible from its 2-neighbourhoods.
\end{theorem}

We note that there is still a gap where we do not know whether $\mathcal G$ can be reconstructed with high probability, and it would be interesting to remove this.

\begin{qu*}
    Determine when $\mathcal G(n,p)$ is  2-reconstructible. Is there a threshold around $n^{-3/4}$ (up to a polylogarithmic factor)?
\end{qu*}

\bigskip\noindent{\em Radius 1:}
We finish this section by looking at reconstruction from balls of radius 1.
Gaudio and Mossel \cite{gaudio2020shotgun}  showed that, for any $\eps > 0$, a random graph $G\in\mathcal G(n,p)$ is
1-reconstructible with high probability when $n^{-1/3 + \eps} \leq p \leq n^{-\eps}$; and fails to be 1-reconstructible with high probability when $n^{-1 + \eps} \leq p \leq n^{-1/2 - \eps}$. This was recently improved in an impressive paper of Huang and Tikhomirov \cite{huang2021shotgun} which showed that there are constants $c, C> 0$ such that
$G$ is 1-reconstructible with high probability when $n^{-1/2} \log^{C}n \leq p \leq c$, while $G$ fails to be 1-reconstructible if $p=o(1/\sqrt n)$ and  $p=\omega(\log (n) / n)$. This shows that there is a change of behaviour around $n^{-1/2}$, up to a polylogarithmic gap. We give a small improvement on the region where $G$ fails to be 1-reconstructible: we improve the lower bound, and give a slight sharpening of the upper bound. Note that in particular this shows that some polylogarithmic factor is indeed necessary.

\begin{theorem}\label{thm:r=1} Let $p=p(n)$ and $G\in\mathcal G(n,p)$.
    If $p=\omega(n^{-3/2})$ and $p \leq  \sqrt{\frac{\log n }{25n}}$, then with high probability $G$ cannot be reconstructed from its 1-neighbourhoods.
\end{theorem}

We further show that the lower bound is sharp.

\begin{theorem}\label{cor:r=1}
    Let $p=p(n)$ and $G\in\mathcal G(n,p)$.
    If $p = o(n^{-3/2})$, then with high probability $G$ is reconstructible from its 1-neighbourhoods.
\end{theorem}

We note that, for very sparse graphs, there are results for even larger radii.
Mossel and Ross \cite{mossel2017shotgun}  showed that if
$p= \lambda/n$ with $\lambda \neq 1$, then there are constants $c_\lambda, C_\lambda$ such that $G$ is with high probability $r$-reconstructible if $r\geq C_\lambda \log n$ and with high probability not $r$-reconstructible if $r \leq c_\lambda \log n$. Very recently sharp asymptotics were obtained by Ding, Jiang and Ma \cite{ding2022shotgun} (including for the case $\lambda = 1$).
\bigskip

As with most graph reconstruction problems, the shotgun assembly problem is closely related to the famous {\em reconstruction conjecture} of Kelly and Ulam \cite{kelly1942isometric,kelly1957congruence,ulam1960collection}. The conjecture asserts that every graph $G$ with at least 3 vertices can be determined up to isomorphism from its vertex-deleted
subgraphs (i.e.~from the multiset $\{G-v : v \in V (G)\}$ of unlabelled subgraphs).
There has been substantial work by many different authors over many years on this conjecture (see e.g. \cite{bondy1977graph,bondy1991graph,asciak2010survey, lauri2016topics} for surveys and background),
and on variants with less information such as using fewer subgraphs (see e.g. \cite{myrvold1988ally,myrvold1989ally,myrvold1990ally,bollobas1990almost,molina1995correction,bowler2010families}) and smaller subgraphs (see e.g. \cite{giles1976reconstructing,muller1976probabilistic,kostochka20193,spinoza2019reconstruction,groenland2021reconstructingtrees}). M\"uller 1976 \cite{muller1976probabilistic} and Bollob\'as 1990 \cite{bollobas1990almost}  showed that the conjecture holds for almost all graphs. In fact, they showed that for reconstructing a random graph one needs significantly less information, for example, only a few of the vertex-deleted subgraphs are needed. The shotgun assembly problem can thus be viewed as a variant of the reconstruction problem using just \emph{local} information.

\bigskip
The paper is organised as follows. In the next section, we give a brief discussion of our proof techniques, and state some probabilistic lemmas that we will use throughout the rest of the paper. In \Cref{sec:mains} we give skeleton proofs for \Cref{thm:r=3,thm:r>3}, breaking the full proof into a series of (technical) claims that will be proved in \Cref{sec:claims}. In \Cref{sec:2nghb} we prove \Cref{thm:r=2}, and in \Cref{sec:12non} we prove \Cref{thm:r=2-non} and \Cref{thm:r=1}.

\section{Discussion and definitions}\label{sec2}

In this section we give short descriptions of some of the main ideas in our proofs.
We will use a very simple but powerful tool for reconstructing graphs, known as the `overlap method', which was introduced in the paper of Mossel and Ross \cite{mossel2017shotgun}.
Intuitively, it seems reasonable that if the neighbourhoods of different vertices are very different from each other, then one might be able to identify vertices in the neighbourhoods of other vertices and reconstruct the graph.
In $N_r(v)$ we can see the entire $(r-1)$-neighbourhood of the neighbours of $v$, so if all the $(r-1)$-neighbourhoods are unique, then we can identify the neighbours of $v$ from its $r$-neighbourhood. This leads to the following lemma.

\begin{lemma}[{\cite[Lemma 2.4]{mossel2017shotgun}}]\label{lem:overlap}
    Suppose that a graph $G$ has unique $(r-1)$-neighbourhoods. Then it is reconstructible from its $r$-neighbourhoods.
\end{lemma}

We will use this lemma when we prove reconstructibility in the proofs of \Cref{thm:r=3}(iii) and \Cref{thm:r>3}(iii). However, proving the uniqueness of neighbourhoods is not always a simple task, especially for such a large range of $p$. Moreover, for large values of $r$, we will not have uniqueness of $(r-1)$-neighbourhoods for the entire range of $p$ we consider and we cannot apply the method as is. Instead, we will use the idea of the overlap method to handle high-degree vertices and then apply a different argument for low degree vertices.

Reconstructibility below the first phase transition, that is reconstructibility when $p=o(n^{-\frac{2r+1}{2r}})$, will follow easily from the fact that all components are with high probability small enough to be fully contained in balls of radius $r$ and for us to recognise this.

For showing non-reconstructibility, we need to prove that with high probability there is a second graph $H$ which is not isomorphic to $G$ but has isomorphic $r$-neighbourhoods.
When considering smaller values of $p$, that is, closer to the first phase transition, our reasoning for non-reconstructibility will lie in the small components. Indeed, for such values of $p$ there will be components that are paths with $2r + 1$ vertices with high probability. The non-reconstructibility will follow from the fact that the collection of $r$-neighbourhoods of two disjoint copies of $P_{2r + 1}$ (a path with $2r + 1$ vertices) is isomorphic to the collection of $r$-neighbourhoods of disjoint copies of $P_{2r}$ and $P_{2r+ 2}$, and therefore graphs containing these cannot be uniquely identified. Interestingly, for $r\geq 4$ being non-reconstructible coincides with the existence of these small components, and the second threshold for reconstructibility is around the point where we stop seeing two disjoint copies of $P_{2r+1}$ as components.
For $r\leq 3$ however, a different phenomena occurs and with high probability it is not possible to reconstruct $G$ even after the disappearance of these small paths. Roughly speaking, it turns out that (with high probability) we can find two edges $uv$ and $xy$, where the $(r-1)$-neighbourhoods of the end vertices are isomorphic, but the $r$-neighbourhoods are not. We can replace the edges by $uy$ and $xv$ to get a graph with the same collection of $r$-neighbourhoods, but which is in a different isomorphism class. This property will continue beyond the existence of two isolated copies of $P_{2r + 1}$ for $r\leq 3$, and for $r=3$ it is instead the disappearance of this property which coincides with the second phase transition.

We use the following notation to distinguish between different types of neighbourhoods.
For a vertex $v$, we let $\Gamma_r(v)$ be the set of vertices that are at distance \emph{exactly} $r$ from $v$. We write $|\Gamma_r(v)|$ for the number of such vertices. In the special case that $r=1$ we simply write $\Gamma(v)$ and we use $d(v)=|\Gamma(v)|$ to denote the degree of the vertex $v$.
As mentioned above, we let ${N}_r^{(G)}(v)$ be the graph induced by the vertices at distance at most $r$ from $v$, where the vertices are unlabelled except for the vertex $v$.
We also use $\Gamma_{\leq r}(v)$ to denote the set of vertices of the graph ${N}_r^{(G)}(v)$ (i.e. the vertices at distance at most $r$ from $v$).
In some proofs we will consider subgraphs consisting of neighbourhoods of several vertices and we will give the relevant notation as and when it is needed. 

\begin{remark}\label{rem:efficiency}
    In every case where we prove that the graph $G \in \mathcal{G}(n,p)$ is $r$-reconstructible with high probability, we give an algorithm that reconstructs $G$ provided it has certain properties and prove that a random graph satisfies these properties with high probability. With minor modifications, all of these algorithms can be run in polynomial time.
\end{remark}

\begin{remark}\label{rem:exact}
    One can also consider exact reconstructibility.  A graph $G$ is said to be \emph{exactly reconstructible} from its $r$-neighbourhoods if $G$ is the unique labelled graph with its collection of $r$-neighbourhoods, i.e. for any $H$ such that ${N}_r^{(G)}(v) \simeq {N}_r^{(H)}(v)$ for every $v \in V(G)$, we have $H = G$. \Cref{lem:overlap} holds for exact reconstructibility, but not all reconstructible graphs are exactly reconstructible. For example, any graph with two disjoint edges as components cannot be reconstructed exactly from its neighbourhoods. In particular,  this means there is some $\alpha > 0$ such that $\mathcal{G}(n,p)$ is not exactly reconstructible with high probability when $p$ is both $\omega(1/n^2)$ and at most $\alpha \log (n)/n$. This contrasts with Theorems \ref{thm:r=3}(i), \ref{thm:r>3}(i), \ref{thm:r=2-lower} and \ref{cor:r=1} which show that $\mathcal{G}(n,p)$ is reconstructible for some of this range.
    When $p \leq 1/2$ and $p = \omega(\log^4 (n)/ (n \log \log n))$, the degree neighbourhoods of vertices are unique with high probability \cite{czajka2008improved}.  When this is true, exact reconstructibility from $r$-neighbourhoods is the same as non-exact reconstructibility for all $r \geq 2$. It follows that, when $p \leq 1/2$, we have exact reconstructibility in \Cref{thm:r=2}. A minor adaption of the proof of \Cref{thm:r=3}(iii) would give exact reconstructibility.
\end{remark}

\subsection{Useful facts}\label{sec:prob}

In this section we state some well known probabilistic bounds which will be useful later in the paper.
We start by stating a simple fact about the median(s) of the binomial distribution.
\begin{fact}
    \label{lem:median}
    Let $X \sim \Bin(n,p)$. Then $\prob{X > \ceil{np}} \leq 1/2$.
\end{fact}

We will make frequent use of the following well-known bounds on the tails of the binomial distribution, known as Chernoff bounds (see e.g.
\cite{alon2016probabilistic,janson2011random,mitzenmacher2017probability}).

\begin{lemma}[Follows from Theorem 4.4 in \cite{mitzenmacher2017probability}]
    \label{lem:chernoff}
    Let $X \sim \Bin\left( n, p\right)$, $\mu=np$ and $\eps > 0$. Then
    \begin{align*}
        \bP \left( X \geq ( 1 + \eps) np \right) & \leq \exp \left(  - \frac{\eps^2 \mu} {2 + \eps} \right), \\
        \bP \left( X \leq ( 1 - \eps) np \right) & \leq \exp \left(  - \frac{\eps^2 \mu} {2} \right).
    \end{align*}
\end{lemma}

We will also be interested in tail bounds for binomial distributions where $\mu \to 0$ as $n \to \infty$, for which we use the following simple bound.

\begin{lemma}
    \label{lem:small-mean}
    Let $X \sim \Bin(n,p)$ and $k \in \mathbb{N}$. Then
    \[ \prob{X  \geq k} \leq e (np)^k.\]
\end{lemma}
\begin{proof}
    We have \[\prob{X  \geq k} = \sum_{j = k}^n \binom{n}{j} p^j (1-p)^{n-j}
        \leq \sum_{j = k}^n \frac{n^j}{j!} p^j
        \leq (np)^k \sum_{j=0}^\infty \frac{1}{j!},\] and the result is immediate.
\end{proof}

We will also want to bound the probability that a binomial (or Poisson binomial) random variable takes a specific value, and we now give several useful lemmas bounding these probabilities. The first, due to Rogozin \cite{rogozin1961estimate}, bounds the probability of a mode of independent discrete random variables.

\begin{theorem}[Theorem 2 in \cite{rogozin1961estimate}]
    \label{thm:sup}
    Let $X_1, \dots, X_n$ be a sequence of independent discrete random variables, and let $S = X_1 + \dotsb + X_n$. Let $p_i = \sup_x\bP\left(X_i = x\right)$.
    Then \[\sup_x \bP \left( S = x \right) \leq \frac{C}{\sqrt{\sum_{i=1}^n (1-p_i)}}\] where $C$ is an absolute constant.
\end{theorem}

The following estimate can be derived from the proofs of Theorem 1.2 and Theorem 1.5 in \cite{bollobas2001random}.

\begin{theorem}\label{thm:DML}
    Suppose $X \sim \Bin(n, p)$ where $p = p(n)$ may depend on $n$. Let $q = 1 - p$ and define $\sigma(n)$ by $\sigma = \sqrt{p q n}$. If $\sigma \to \infty$ as $n \to \infty$, then uniformly over all $0 \leq h \leq \sigma^{5/4}$ such that $pn + h \in \mathbb{Z}$, we have
    \[\prob{X = pn + h} = (1 + o_\sigma(1)) \frac{1}{\sqrt{2 \pi \sigma^2}} \expb{ - \frac{h^2}{2 \sigma^2}}.\]
\end{theorem}

In the proof of \Cref{thm:r=2}, we will approximate the sum of Bernoulli random variables with a Poisson random variable for which we use the following result. The first version of this result was given by Le Cam \cite{le1960approximation} in 1960, but there are now several variations and different proofs, and we refer the reader to \cite{steele1994cam} for more discussion. We will use the following version.

\begin{theorem}[Le Cam's Theorem]
    \label{thm:le-cam}
    Let $X_1, \dots, X_n$ be independent Bernoulli random variables with success probabilities $p_1, \dots, p_n$. Let $S  = X_1 + \dotsb + X_n$ and let $\mu$ denote the expectation of $S$ (i.e.
    \(\mu = \expec{S} = \sum_{i=1}^n p_i\)).
    Then
    \[ \sum_{k=0}^\infty \left| \prob{S = k} - \frac{\mu^k e^{-\mu}}{k!} \right| < 2 \min \left\{1, \frac{1}{\mu}\right\} \sum_{i=1}^n p_i^2. \]
\end{theorem}

\section{Reconstruction from \texorpdfstring{$r$}{r}-neighbourhoods, \texorpdfstring{$r\geq 3$}{r >= 3}}\label{sec:mains}

In this section we use a series of lemmas to prove \Cref{thm:r=3} and \Cref{thm:r>3}, but we delay proving the more complicated lemmas until the later sections. Both of these proofs employ different arguments for different ranges of $p$, although the proofs of parts (i) and (ii) are very similar in both cases.

We start by recording some simple facts about the structure of random graphs.

\begin{lemma}\label{lem:maxPath}
    Let $r=r(n) \geq 1$ and suppose that $p = p(n) = o(n^{-\frac{2r+1}{2r}})$. Then with high probability a random graph $G\in \mathcal G(n,p)$ does not contain a copy of the path on $2r + 1$ vertices.
\end{lemma}

\begin{proof}
    There are at most $n^{2r+1}$ ordered tuples of $2r+1$ vertices and the probability these form a path (in the given order) is $p^{2r}$. Hence, the probability that there is a path of length $2r+1$ in $G$ is $o(1)$ by Markov's inequality.
\end{proof}

\begin{lemma}\label{claim:2IdenticalPaths}\label{halfway}
    There exists an $\alpha > 0$ such that the following holds for all $1\le r = o(\log n)$. If $p$ is such that ${p}{n^{\frac{2r+1}{2r}}}=\omega(1)$ and $p \le \alpha \frac{\log n}{rn}$, then $G\in \mathcal G(n,p)$ contains two paths of $2r+1$ vertices as components with high probability.
\end{lemma}

\begin{proof}
    Fix $\alpha<1/6$, and let $X$ be the number of path components with $2r+1$ vertices.  The expectation of $X$ is
    \[f(r,n, p):=\frac12\binom{n}{2r+1}(2r+1)!p^{2r}(1-p)^{(2r+1)(n-2r-1)+\binom{2r+1}{2}-2r}.\]
    We may assume that $r \leq \beta \log n$ and $p\ge \lambda n^{-\tfrac{2r+1}{2r}}$ where $\beta = \beta(n)$ and $\lambda=\lambda(n)$ are functions that slowly tend to 0 and infinity respectively. For fixed $n$ and $r$, the function $f$ (as a function of $p$) has the form $f(p)=Cp^a(1-p)^b$ for some positive constants $C,a,b$. When $p\in [0,1]$ this function is 0 at the endpoints of the interval, and positive otherwise. It is also easy to check that the function obtains a single maximum  in $[0,1]$. Thus the minimum of $f(r, n , p)$ over $p \in [\lambda n^{-\frac{2r+1}{2r}},\alpha \frac {\log n}{rn}]$ is attained at one of the end points. 

    We have that 
    \(f(r, n, p) \geq \frac12 (n - 2r)^{2r+1} p^{2r} (1-p)^{(2r + 1)n}\), and so substituting in $p_0 = \lambda n^{-\frac{2r+1}{2r}}$ and using that $1-x \geq e^{-2x}$ for small $x$, we find that 
    \begin{align*}
        f(r, n, p_0) &\geq \frac{1}{2} \left(\frac{n - 2r}{n}\right)^{2r+1} \exp\left( 2r \log\lambda - 2(2r+1) \lambda n^{-1/2r}\right)\\
        &\geq \frac{1}{2} \left(\frac{n - 2r}{n}\right)^{2r+1} \exp\left( 2r \left(\log\lambda - 3 \lambda \exp(-1/(2\beta))\right)\right).
    \end{align*}
    This is $\omega(1)$ provided that $\lambda$ 
    grows sufficiently slowly compared to $1/\beta$.    
    Similarly, substituting in $p_1 = \alpha \frac{\log n}{rn}$ we find that
    \begin{align*}
        f(r, n, p_1) &\geq \frac{1}{2} \left(\frac{n - 2r}{n}\right)^{2r+1} n \exp\left(2r \log(\alpha \log(n) / r ) - 2 \frac{2r+1}{r} \alpha \log n \right)\\
        &\geq  \frac{1}{2} \left(\frac{n - 2r}{n}\right)^{2r+1} \exp\left((1-6\alpha) \log n + 2 r \log (\alpha \log  (n)/r) \right),
    \end{align*}
    which is $\omega(1)$ provided $\alpha < 1/6$.
    Hence, $\expec{X} \to \infty$ as $n \to \infty$.
    
    We now bound $\E [X^2]$.    Let $\gamma$ be the probability that a specific set of $2r+1$ vertices induces a path component.
    Note that distinct components cannot share vertices, so $\E[X^2]$ decomposes as $\E[X]$ plus a sum over disjoint pairs of $(2r+1)$-sets.  The probability that two specific disjoint sets of $2r+1$ vertices both induce path components is $\gamma^2(1-p)^{-(2r+1)^2}$, as there are $(2r+1)^2$ potential edges between the sets.  Since $(1-p)^{-(2r+1)^2} = (1 + o(1))$, we find that $\E[X^2]=(1+o(1))\E[X]^2$.
    By Chebyshev's inequality, we obtain that with high probability $X\ge 2$.
\end{proof}

Combining the two lemmas above gives the following lemma, which handles the first phase transition.

\begin{lemma}\label{thm:firstPhaseTransition}
    Let $G\in \mathcal G(n,p)$.
    There is a constant $\alpha > 0$ such that, for all $1 \leq r = o(\log n)$,
    \begin{equation*}
        \lim_{n \to \infty}\prob{G\text{ is $r$-reconstructible}} =
        \begin{cases}
            1, & \text{if $p=o\left(n^{-\frac{2r+1}{2r}}\right)$,}                                            \\
            0, & \text{if $p=\omega\left(n^{-\frac{2r+1}{2r}}\right)$ and $p\leq \alpha \frac {\log n}{rn}$.}
        \end{cases}
    \end{equation*}
\end{lemma}

\begin{proof}
    The dense regime follows immediately from \Cref{claim:2IdenticalPaths} and the fact that the graph consisting of two paths of $2r+1$ vertices is not reconstructible (see \Cref{sec2}).

    For the sparse regime, we note first that if a graph has no path of length $2r + 1$, then each component must be contained in the $r$-ball around one of its vertices. Indeed, if this is not the case, then the radius of the component must be at least $r + 1$ and the component contains an (induced) path with $2r + 1$ vertices \cite{erdos1986maximum}. If the graph does contain a path with at least $2r + 1$ vertices, then there must be an $r$-ball containing a path with at least $2r + 1$ vertices.

    Suppose there is no $r$-ball containing a path with at least $2r + 1$ vertices. Then we start by choosing an $r$-ball with as many vertices as possible: this gives us an entire component $C$, and from this we can determine the $r$-balls of all vertices in $C$.  We now delete all these $r$-balls from our collection, and repeat on the remaining $r$-balls (which are exactly the $r$-balls of $G$ with $C$ deleted). This will reconstruct the graph $G$, and the claim follows since \Cref{lem:maxPath} implies that no $r$-ball has a path on $2r + 1$ vertices with high probability.
\end{proof}

We remark that the algorithm in the proof above runs in polynomial time when $r = o(\log n)$. 
First, we need to check that there are no paths of length $2r$. This can be done in time $2^{O(r)} n \log n$ \cite{alon1995color}, and this is polynomial in $n$ if $r = O(\log n)$. The other key step is determining the $r$-balls of all vertices in $C$ and deleting all these $r$-balls from our collection, for which we may need to solve the graph isomorphism problem (a polynomial number of times). Fortunately, this can be done in quasipolynomial time \cite{babai2016graph} in the number of vertices and we only need to compare graphs with $o(\log n)$ vertices, so the total time is polynomial in $n$.

The following lemma will be useful when proving \Cref{thm:r>3}(iii).

\begin{lemma}\label{claim:CompSizeHighDegVxs}
    There exists $\beta > 0$ such that the following holds for all $4 \leq r \leq \log n$, and $p\ge \beta \tfrac{\log n}{rn}$. Let $G\in \mathcal G(n,p)$, and let $H$ be the subgraph of $G$ induced by the vertices with degree at most $np/2$. Then with high probability the maximum component size of $H$ is at most $r-3$.
\end{lemma}

\begin{proof}
    Fix $\beta > 5$ such that $\log\beta - \beta/9 + 1 \leq -\beta/10$, e.g. $\beta = 677$. It is enough to bound the probability of the event $E$ that there is a set $A$ of $r-2$ vertices such that $G[A]$ is connected and each vertex in $A$ has at most $np/2$ neighbours outside $A$. For fixed $A$, these two properties are independent, and we bound the probability of each property as follows. If $G[A]$ is connected, then it must contain a spanning tree. Any particular spanning tree is present with probability $p^{r-3}$ and there are $(r-2)^{r-4}$ possible spanning trees, so the probability that $G[A]$ is connected is at most $p^{r-3} (r-2)^{r-4}$. 
    Let $X\sim \Bin (n-r+2,p)$. Then the probability that $v\in A$ has at most $np/2$ neighbours outside $A$ equals $\prob{X\leq np/2}$, which by a Chernoff bound (\Cref{lem:chernoff})
    is at most $e^{-np/9}$ for large enough $n$.
    
    There are $\binom{n}{r-2} \leq ( \frac{e n}{r-2})^{r-2}$ possible choices for the set $A$, so we can upper bound the probability that $E$ occurs by 
    \begin{multline*}
                p^{r-3}(r-2)^{r-4} \cdot e^{-(r-2)np/9} \cdot \left( \frac{e n}{r-2}\right)^{r-2} = \\
                \exp\left( (r-2) \left(\log(np) - \frac{1}{9} np + 1\right) - \log p - 2 \log(r-2)\right).
    \end{multline*}
        
    Now we use that $r \geq 4$ and the way we have chosen $\beta$ to get the upper bound
    \[
        \bP(E) \leq \exp\left(- \frac{\beta}{20} \log n  + \log n\right),
    \]
    which clearly tends to $0$ as $n \to \infty$. 
\end{proof}

We will also need several facts about small balls in random graphs.  The proofs of these are more complicated so we postpone them to  \Cref{sec:claims}.

\begin{lemma}\label{clauniq2}
    For any $\eps > 0$, there exists $\beta > 0$ such that, for  $\beta\tfrac{\log^2 n}{n\left(\log\log n\right)^3} \le p \le n^{-2/3 - \eps}$, the $2$-neighbourhoods of $G\in \mathcal G(n,p)$ are unique with high probability.
\end{lemma}

\begin{lemma}\label{clauniq3}
    Suppose $\tfrac{\log^{2/3} n}{n} \le p\le \frac{\log^2 n}{n}$. Then, with high probability, there are no two vertices $x,y$ of $G\in \mathcal G(n,p)$ with degree at least $np/2$ such that the $3$-neighbourhoods around $x$ and $y$ are isomorphic (i.e. the $3$-neighbourhoods around vertices with degree at least $np/2$ are unique).
\end{lemma}

\begin{lemma}\label{claswapr=3}
    Let $\alpha > 0$ be a sufficiently small constant and suppose $\tfrac{\log^{2/3} n}{n} \le p\le \alpha \tfrac{\log^2 n}{n(\log\log n)^3}$.
    Then, for $G\in\mathcal G(n,p)$, with high probability there are distinct vertices $x, y, u, v$ such that $xy,uv\in E(G)$ and $xv,yu\notin E(G)$ and the graph $G'$ obtained from $G$ by deleting  $xy,uv$ and adding $xv,yu$ satisfies the following:
    \begin{enumerate}
        \item $G$ and $G'$ are not isomorphic.
        \item $G$ and $G'$ have the same collection of $3$-balls.
    \end{enumerate}
\end{lemma}

We now piece together the lemmas above to give proofs of \Cref{thm:r=3} and \Cref{thm:r>3}.

\begin{proof}[Proof of  \Cref{thm:r=3}(i) and  \Cref{thm:r>3}(i)]
    Follows immediately from \Cref{thm:firstPhaseTransition}.
\end{proof}

\begin{proof}[Proof of \Cref{thm:r=3}(ii) and \Cref{thm:r>3}(ii)]
    \Cref{thm:r>3}(ii) follows immediately from\\ \Cref{thm:firstPhaseTransition}, but the lemma does not give the entire range of $p$ needed in \Cref{thm:r=3}(ii),  and we will use a different argument for larger $p$. To cover the remaining region, it is enough to show that there exists $\alpha>0$ such that $G\in \mathcal G(n,p)$ is not reconstructible from its 3-neighbourhoods with high probability when $\frac{\log^{2/3} n}{n}\leq p\leq \alpha \frac{\log^2n}{n(\log\log n)^3}$, and this is exactly the content of \Cref{claswapr=3}.
\end{proof}

\begin{proof}[Proof of \Cref{thm:r=3}(iii)]
    \Cref{thm:r=2} shows there is a constant $\delta > 0$ such that the graph can be reconstructed from its 2-neighbourhoods with high-probability when $p \geq n^{-2/3 - \delta}$. Hence, we can assume that $\beta \frac{\log^2n}{n (\log \log n)^3} \leq p \leq n^{-2/3 - \delta/2}$, and it follows from \Cref{clauniq2} that the $2$-neighbourhoods are unique with high probability. The result now follows immediately by applying \Cref{lem:overlap}.
\end{proof}

\begin{proof}[Proof of \Cref{thm:r>3}(iii)]
    By \Cref{thm:r=3}(iii), $G\in\mathcal G(n,p)$ is reconstructible with high probability from its $3$-neighbourhoods when $p = \Omega(\log^2(n)/ ( n \log \log n ))$, so we may assume that $p = O(\log^2 (n)/n)$. We use the overlap method to reconstruct the portion of the graph induced by vertices of moderately large degree; a further argument is needed to reconstruct the rest of the graph.

    Let $V_1$ be the vertices of $G$ with degree at least $np/2$ and let $V_2=V(G)\setminus V_1$. For $i=1,2$, let $H_i$ be the subgraph induced by $V_i$.  For each vertex $v$, we can determine from its 1-ball whether $v\in V_1$ or $v\in V_2$. When the 3-balls (in $G$) around the vertices in $V_1$ are unique, we can easily reconstruct $H_1$ using the overlap method, and this event happens with high probability by \Cref{clauniq3}.

    Now consider $H_2$.
    By \Cref{claim:CompSizeHighDegVxs} we may assume that all components of $H_2$ have at most $r-3$ vertices, and note that we can easily check that this holds from the $r$-balls. Consider a component $C$ of $H_2$.  For each vertex $v$ of $C$, the $(r-4)$-ball around $v$ contains all vertices of $C$, so the $(r-3)$-ball contains all vertices of $V_1$ that are adjacent to a vertex of $C$. The $r$-ball around $v$ contains the 3-balls around the vertices in $V_1$ that are adjacent to a vertex
    of $C$, and we assume that these are all unique.  It follows that by looking at the $r$-ball around $v$, we can identify $C$ (up to isomorphism), and for each vertex of $C$, we can determine which vertices of $V_1$ it is adjacent to.  We obtain this information $|C|$ times for each component $C$ of $H_2$ (once for each vertex of $C$), and so allowing for multiplicities we can reconstruct all components of $H_2$ and the way they are attached to $H_1$.
\end{proof}

The two proofs above both give algorithms to (attempt to) reconstruct a graph from its $r$-neighbourhoods, although they do not necessarily run in polynomial time. Both of these algorithms use the overlap method which requires checking if the $(r-1)$-neighbourhood of a vertex in one neighbourhood is the same up to isomorphism as the $(r-1)$-neighbourhood of the marked vertex in a different neighbourhood, and these neighbourhoods could have polynomially many vertices.
However, we can weaken the overlap method slightly and instead require that the $(r-1)$-neighbourhoods are more obviously distinct. For example, in the proof of  \Cref{thm:r=3}(iii) we could require that the multiset of degrees of the neighbours of each vertex is unique. This is in fact how we prove \Cref{clauniq2}, and so the result still holds, but these multisets can be compared in polynomial time. 

For a vertex $v$, let $D(v)$ be the multiset of degrees of the neighbours of $v$. For the proof of \Cref{thm:r>3}(iii), we label the vertex $u$ by the multiset $\{D(v) : v \in \Gamma(u)\}$. It is easy to compare the labels of the vertices in polynomial time, and the proof of \Cref{clauniq3} shows that no two vertices with degree at least $np/2$ have the same label. The proof also requires that we check the isomorphism class of the components of $H_2$, but we assume these all have $o(\log n)$ vertices. 

We note that both the proof of \Cref{thm:r=3}(iii) and the proof of \Cref{thm:r>3} make use of \Cref{thm:r=2}, but the proof of this theorem also implicitly gives a polynomial algorithm.

\begin{remark}
    Simultaneous work of Gaudio, R\'{a}cz and Sridhar \cite{gaudio2022average} also proved a result on the uniqueness of 3-balls over a different range of $p$. They proved the stronger result that the 3-balls around \emph{all} of the vertices are non-isomorphic, not just those around the vertices of degree at least $np/2$. However, they require $(1+\delta) \log(n)/n \leq p \leq 1/2$, and their result is not sufficient for our use here. In fact, such a result cannot hold for the smaller values of $p$ that we require as there will be many isolated vertices with isomorphic 3-balls.
\end{remark}

\section{Reconstruction from 2-neighbourhoods}\label{sec:2nghb}

In this section we prove \Cref{thm:r=2}. Since Gaudio and Mossel \cite{gaudio2020shotgun} proved that, for all $\eps > 0$, a random graph $G\in \mathcal G(n,p)$ can be reconstructed from its collection of $2$-balls if $n^{-1/2+\eps}\leq p$ with high probability, we may assume that $p \leq n^{-16/35}$.

We use an approach similar to that of Gaudio and Mossel \cite{gaudio2020shotgun}. We will colour each edge $uv$ by a colour which can be determined from the 2-neighbourhoods of both $u$ and $v$ and we attempt to reconstruct the graph from the edge-coloured stars around the vertices. Gaudio and Mossel \cite{gaudio2020shotgun} showed that this information is sufficient to reconstruct an edge-coloured graph when no two edges have the same colour. In order to prove our result, we will use colourings which satisfy a slightly weaker condition which is easier to show.

\begin{lemma}\label{lem:edge-coloured}
    Let $G$ be an edge-coloured graph such that every pair of edges of the same colour share a vertex. Then by looking only at the number of edges of each colour adjacent to each vertex, $G$ can be reconstructed exactly.
\end{lemma}

\begin{proof}
    Let our edge-coloured stars be $S_1,\dots,S_n$, and label the corresponding centres $v_1,\dots,v_n$.
    Fix a colour $c$ and consider the subgraph $H$ consisting of all edges with this colour. From the degree sequence of $H$ we can check if $H$ (up to isolated vertices) is a triangle or a star, and note that these are the only graphs with no disjoint edges so $H$ must be one of these graphs. In either case, we can reconstruct $H$ by joining $v_i$ and $v_j$ with an edge in colour $c$ whenever one of $v_i$ and $v_j$ is a vertex of largest degree in colour $c$ (and they are both incident to at least one edge coloured with $c$).  The graph $G$ is the union (over all colours) of these subgraphs.
\end{proof}

We now give the edge colouring we will use and show that with high probability no two disjoint edges have the same colour. For an edge $uv$, let $C_{uv}$ be the subgraph of $G$ induced by the vertices at distance at most 2 from both $u$ and $v$, where we distinguish the edge $uv$. We write $C_{uv} \simeq C_{xy}$ if there is a bijection $f:  V(C_{uv}) \to V(C_{xy})$ such that $ab \in E(C_{uv})$ if and only if $f(a)f(b) \in E(C_{xy})$, and $\{f(u),f(v)\} = \{x,y\}$. We will refer to each such isomorphism class as a \emph{colour}.
\Cref{thm:r=2} follows immediately from \Cref{lem:edge-coloured} and the following.

\begin{lemma}
    \label{lem:2-neighbourhoods}
    There exists a constant $\delta > 0$ such that the following holds. Suppose  $n^{-2/3 - \delta} \leq p \leq n^{-16/35}$, and let $u,v,x,y$ be distinct vertices.  The probability that
    $uv$ and $xy$ are edges, and $C_{uv}\simeq C_{xy}$ is $o(n^{-4})$.
\end{lemma}

Before proving \Cref{lem:2-neighbourhoods}, we explain how it implies \Cref{thm:r=2}.

\begin{proof}[Proof of \Cref{thm:r=2}]
    For each edge $uv$ in $G\in\mathcal G(n,p)$, we colour the edge $uv$ with the isomorphism class of $C_{uv}$, and note that for each vertex $u$ it is possible to determine the colour of all edges incident with $u$ from the 2-ball around $u$. Indeed, if $x$ is a vertex at distance at most 2 from $u$ and $v w x$ is a path from $v$ to $x$, then $v$, $w$ and $x$ are all contained in the 2-ball around $u$. This means we can determine which vertices in the 2-ball around $u$ are also in the 2-ball around $v$, and we can determine the isomorphism class of $C_{uv}$. It follows from \Cref{lem:2-neighbourhoods} that with high probability no two disjoint edges have the same colour, and by \Cref{lem:edge-coloured}, we can then reconstruct $G$.
\end{proof}

\begin{figure}
    \centering
    \begin{tikzpicture}
        \draw (-1,0) -- (1,0);
        \draw (1 - 0.489898, -1 + 0.04) -- (1,0) -- (1 +0.489898, -1 + 0.04);
        \draw (-1 - 2*0.489898, -2 + 0.08) -- (-1,0) -- (-1 + 2*0.489898, -2 + 0.08);

        \draw[fill] (1,0) circle (.05);
        \node[above] at (1,0) {$v$};
        \draw[fill] (-1,0) circle (.05);
        \node[above] at (-1,0) {$u$};
        \draw  (-1,-1) ellipse (0.5 and 0.2);
        \draw  (-1,-2) ellipse (1 and 0.4);
        \draw  (1,-1) ellipse (0.5 and 0.2);

        \draw[fill,black!60!green] (0.8,-1) circle (.05);
        \draw[black!60!green] (0.8, -1) -- (-0.4, -2);
        \draw[black!60!green] (0.8, -1) -- (-0.6, -2);
        \draw[black!60!green] (0.8, -1) -- (-0.8, -2);
        \draw[black!60!green] (0.8, -1) -- (-0.2, -2);

        \draw (-1.3, -1) -- (-1,0);
        \draw (-1.3, -1) -- (1,0);
        \draw[fill, red] (-1.3,-1) circle (.05);
        \draw[red] (-1.3,-1) -- (-1.4, -2);
        \draw[red]  (-1.3,-1) -- (-1.6, -2);
        \draw[red]  (-1.3,-1)-- (-1.8, -2);
        \draw[red]  (-1.3,-1) -- (-1.2, -2);
        
    \end{tikzpicture}
    \caption{We will show the $C_{uv}$ are unique by considering the number of edges each vertex in $\Gamma_1(v)$ has to $\Gamma_2(u) \setminus \Gamma_1(v)$. The vertex adjacent to $u$ and $v$ shown in red will be problematic and we will view its degree as an ``error''.}
    \label{fig:problematic}
\end{figure}
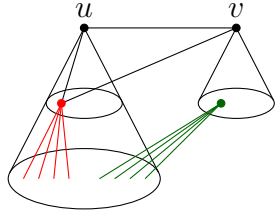

Before giving the full details of the proof of \Cref{lem:2-neighbourhoods}, let us sketch our strategy.
Suppose that $C_{uv}$ and $C_{xy}$ are isomorphic with $u$ mapping to $x$ and $v$ mapping to $y$. Then it must be the case that the unordered degree sequence of $\Gamma_1(v)$ into $\Gamma_2(u) \setminus \Gamma_1(v)$ and of $\Gamma_1(y)$ into $\Gamma_2(x) \setminus \Gamma_1(y)$ are equal, and we will show that the probability of this event is $o(n^{-4})$. We note that although we cannot see the whole of $\Gamma_2(u)$ in $C_{uv}$, we do see all the edges from $\Gamma_1(v)$ to $\Gamma_2(u)$ and we can therefore read off the degree sequence of $\Gamma_1(v)$ into $\Gamma_2(u)\setminus \Gamma_1(v)$. By symmetry, the probability of an isomorphism which maps $u$ to $y$ and $v$ to $x$ will also be $o(n^{-4})$.

Fix $u$ and $v$ and suppose that $uv$ is an edge. We reveal the edges from $u$ and $v$, and then from $\Gamma_1(u)$. Given a vertex $i$ in $\Gamma_1(v) \setminus \Gamma_1(u)$ which is not $u$, we have not revealed any of its edges to $\Gamma_2(u) \setminus \Gamma_1(v)$ so the number of such edges $b(i)$ is a binomial random variable with $|\Gamma_2(u)|$ trials and success probability $p$. When $p$ is only a little bit bigger than $n^{-2/3}$, we have $|\Gamma_2(u)| = \Theta(n^2p^2)$ and $b(i)$ takes each of the $np^{3/2}$ most likely values with probabilities which are $\Theta(n^{-1}p^{-3/2})$. If we ignore problematic vertices (see \Cref{fig:problematic} for an example of a problematic vertex) and assume that every vertex is an independent binomial, the number of vertices $N_k$ in $\Gamma_1(v)$ with a fixed likely degree $k$ is a binomial random variable with $\Theta(np)$ trials and success probability $\Theta(n^{-1}p^{-3/2})$. We also do the same thing for the edge $xy$ to find that the probability that there are $N_k$ vertices in $\Gamma_1(y)$ with $k$ edges to $\Gamma_2(x) \setminus \Gamma_1(y)$ is $O(p^{+1/4})$. By considering multiple values of $k$, we can show that the probability that $C_{uv}$ is isomorphic to $C_{xy}$ is $o(n^{-4})$.

Unfortunately, this sketch has glossed over many details, most notably the dependencies between the different values we consider, and we will have to work considerably harder to make the argument rigorous. At various points we will see different behaviour for different values of $p$ in the range (e.g. the number of vertices in $\Gamma_2(u)$ is not $\Theta(n^2p^2)$ when $p = \omega(n^{-1/2})$), and we will have to employ different arguments for different ranges of $p$.

Finally, we remark that our proof actually gives an efficient algorithm for reconstructing a random graph $G\in\mathcal G(n,p)$ from its 2-neighbourhoods which succeeds with high probability. Instead of colouring the edge $uv$ by the isomorphism class of $C_{uv}$, we can colour it by a combination of the unordered degree sequence of $\Gamma_1(v)$ into $\Gamma_2(u)\setminus \Gamma_1(v)$ and the unordered degree sequence of $\Gamma_1(u)$ into $\Gamma_2(v) \setminus \Gamma_1(u)$. The proof of \Cref{lem:2-neighbourhoods} shows that any two disjoint edges get the same colour with probability $o(n^{-4})$, and  \Cref{lem:edge-coloured} applies with high probability. These degree sequences can clearly be calculated efficiently.

\begin{proof}[Proof of \Cref{lem:2-neighbourhoods}]
    Fix four vertices $u$, $v$, $x$ and $y$, and condition on the event that $uv$ and $xy$ are edges. Let $M$ be the set of vertices which are adjacent to at least $2$ of the vertices in  $\{x, y, u,v\}$. These vertices introduce dependence between the degree sequences we care about, and we will view these vertices as introducing an ``error'' of size at most $|M|$. We are therefore interested in an upper bound for $|M|$. There are 6 pairs of vertices from $\{x, y, u, v\}$ and the probability that a vertex is adjacent to a given pair is $p^2$, so $|M|$ is dominated by a $\Bin(n, 6p^2)$ random variable.

    \begin{claim}
        Let \[m = \begin{cases}
                12n^{1/9} & p > n^{-11/20},    \\
                40        & p \leq n^{-11/20}.
            \end{cases}\]
        Then
        \[\prob{|M| > m} = o(n^{-4}).\]
    \end{claim}
    \begin{proof}
        The first case follows almost immediately from the Chernoff bound in \Cref{lem:chernoff}. Indeed, since $p \leq n^{-16/35} \leq n^{-4/9}$, $|M|$ is clearly dominated by a $\Bin(n, 6n^{-8/9})$ random variable, and the probability that this exceeds $12n^{1/9}$ is at most $\exp(-2 n^{1/9}) = o(n^{-4})$.

        The second case follows from \Cref{lem:small-mean}. In this case, $|M|$ is stochastically dominated by a $\Bin(n, 6n^{-11/10})$ random variable and
        \[\prob{|M| \geq 41} \leq e (6n^{-1/10})^{41} = o(n^{-4}).\]
    \end{proof}

    We now look to bound the size of the neighbourhood of a vertex.

    \begin{claim}
        \label{claim:neighbourhoods}
        Fix a vertex $i$, and let
        \[\lambda(i)=(n-1-d(i))(1-(1-p)^{d(i)}).\]
        Then with probability $1-o(n^{-4})$ we have
        \[\frac{np}{2} \leq d(i) \leq 2 np,\]
        and
        \[||\Gamma_2(i)|-\lambda(i)|\leq (np)^{5/4}.\]
    \end{claim}
    \begin{proof}
        The degree of $i$ follows a $\Bin(n-1,p)$ distribution so using a Chernoff bound (see \Cref{lem:chernoff}), the probability that $d(i)$ is less than $np/2$ is at most
        \[ 4\exp\left(-\frac{(n-2)^2p}{8(n-1)}\right) = \exp\left(-\Theta(np)\right) = o(n^{-4}).\]
        In the other direction, the other bound in \Cref{lem:chernoff} shows that the probability $d(i) \geq 2np$ is also at most $4\exp(-np/3) = o(n^{-4})$.

        Given $d(i)$, the size of the second neighbourhood of $i$ is distributed like
        \[X\sim \Bin(n-1-d(i),1 - \left(1-p\right)^{d(i)}),\]
        so $\E[X]=\lambda(i)$. If $\lambda(i) = \omega(\log^8n)$, then \[\prob{|X - \lambda(i)| \geq \lambda(i)^{9/16}} \leq \exp(- \Theta(\lambda(i)^{1/8})) = o(n^{-4}).\] Hence, it suffices to prove that with probability $o(n^{-4})$ we have $\lambda(i) = \omega(\log^8n)$ and (for large enough $n$) $\lambda(i)^{9/16} \leq (np)^{5/4}$.

        For the first statement, we may assume that $np/2 \leq d(i)\leq 2np$. Using that $1 - t \leq e^{-t} \leq 1 - t/2$ for all $t \in [0,1]$, we have
        \begin{align*}
            \lambda(i)
             & = (n - 1 - d(i)) \left(1 - (1 - p)^{d(i)}\right)  \\
             & \geq \tfrac{n}{2}   \left(1 - (1-p)^{np/2}\right) \\
             & \geq \tfrac{n}{2} (1-e^{-np^2/2})                 \\
             & \geq \tfrac{n}{2}\min\{1-e^{-1}, np^2/4\}
        \end{align*}
        for large enough $n$. This is $\omega(\log n)$ in our range of $p$.

        For the second statement, note that $\lambda(i)\le n(1-(1-p)^{2np})\le 2n^2p^2$, by Bernoulli's inequality.
    \end{proof}

    We will shortly reveal the edges from $\Gamma_1(u)$ and from $\Gamma_1(x)$ to discover their second neighbourhoods. Unfortunately, this may reveal some edges from $\Gamma_1(v)$ to $\Gamma_2(u)\setminus \Gamma_1(v)$. For example, if $i \in \Gamma_1(v)$, then we will be revealing all edges from $i$ to $\Gamma_1(x)$. Some of the vertices in $\Gamma_1(x)$ may also be in $\Gamma_2(u) \setminus \Gamma_1(v)$, so we have revealed some of the edges from $i$ to $\Gamma_2(u) \setminus \Gamma_1(v)$.
    We will use the following lemma to control how many edges have been revealed.

    \begin{claim}
        Let $t \in \{u, v, x, y\}$.
        If $n^{-11/20} \leq p \leq n^{-4/9}$, then the probability there exists a vertex $j \not \in \{t\} \cup \Gamma_1(t)$ which is adjacent to at least $\left(n^2p^3\right)^{1/4}$ vertices in $\Gamma_1(t)$ is $o(n^{-4})$.

        If $p \leq n^{-11/20}$, then the probability there exists a vertex $j \not \in \{t\} \cup \Gamma_1(t)$ which is adjacent to at least $51$ vertices in $\Gamma_1(t)$ is $o(n^{-4})$.
    \end{claim}

    \begin{proof}
        Suppose first that $n^{-11/20} \leq p \leq n^{-4/9}$. For a given vertex $j$, the number of neighbours in $\Gamma_1(t)$ is a binomial random variable with $d(t) = |\Gamma_1(t)|$ trials and success probability $p$. We may assume that $d(t)\leq 2np$ and, by applying a Chernoff bound (\Cref{lem:chernoff}), we find that the probability that $j$ is adjacent to at least $\left(n^2p^3\right)^{1/4}$ vertices in $\Gamma_1(t)$ is at most
        \[\exp\left(-\Theta(n^2p^3)^{1/4}\right),\] provided $np^{5/2} \to 0$. There are at most $n$ choices for $j$ and applying a union bound completes the proof.

        To prove the second part of the claim where $p \leq n^{-11/20}$, we use \Cref{lem:small-mean}. For a given vertex $j$, the number of neighbours of $j$ is dominated by a binomial random variable with mean $2np^2 \leq 2n^{-1/10}$. Hence, by \Cref{lem:small-mean}, the probability that a vertex has at least $51$ neighbours in $\Gamma_1(t)$ is $O(n^{-51/10})$. Taking a union bound over all choices for the vertex $j$, the probability that any suitable $j$ is adjacent to at least $51$ vertices from $\Gamma_1(t)$ is $o(n^{-4})$ as required.
    \end{proof}

    We now reveal the edges from $u$, $v$, $x$ and $y$, the edges from $\Gamma_1(u)$ and $\Gamma_1(x)$ and the edges between the neighbours of $u$, $v$, $x$ and $y$. None of the other edges need to be revealed and they are still each present independently with probability $p$. We also check that the following have all occurred and note that each of them occurs with probability $1 - o(n^{-4})$.
    \begin{itemize}
        \item $|M|$ is bounded above by $m$,
        \item $d(u), d(v), d(x)$ and $d(y)$ are all in $[np/2, 2 np]$,
        \item $\left| |\Gamma_2(u)| - \lambda(u)\right| \leq (np)^{5/4}$ and $\left| |\Gamma_2(x)| - \lambda(x)\right| \leq (np)^{5/4}$,
        \item  for every vertex $a \in \Gamma_1(v)$, the sets $\Gamma_1(a) \cap \Gamma_1(u)$, $\Gamma_1(a) \cap \Gamma_1(x)$ and $\Gamma_1(a) \cap \Gamma_1(y)$ have size at most $\left(n^2p^3\right)^{1/4}$ if $n^{-11/20} \leq p \leq n^{-4/9}$, or $51$ if $p \leq n^{-11/20}$, and
        \item for every vertex in $b \in \Gamma_1(y)$, the sets $\Gamma_1(b) \cap \Gamma_1(u)$, $\Gamma_1(b) \cap \Gamma_1(v)$ and $\Gamma_1(b) \cap \Gamma_1(x)$ have size at most $\left(n^2p^3\right)^{1/4}$ if $n^{-11/20} \leq p \leq n^{-4/9}$, or $51$ if $p \leq n^{-11/20}$.
    \end{itemize}

    If there is an isomorphism from $C_{uv}$ to $C_{xy}$ which maps $u$ to $x$, then we must have $d(u) = d(x)$, and we also assume that this event occurs. This means that $\lambda(u) = \lambda(x)$ and we denote the single quantity by $\lambda$.

    Having assumed the above properties, we are ready to begin looking at the the number of edges from each vertex in $\Gamma_1(v)$ to $\Gamma_2(u)\setminus \Gamma_1(v)$ and bound the probability that this unordered degree sequence equals the one from $\Gamma_1(y)$ to $\Gamma_2(x) \setminus \Gamma_1(y)$
    For any $i,j \in V(G)$, let $X_{i,j}$ be the indicator that the edge $\{i,j\}$ is present in $G$, and let \[A = \{x, y, u, v \} \cup \Gamma_1(u) \cup \Gamma_1(v) \cup \Gamma_1(x) \cup \Gamma_1(y).\] For a vertex $i \in \Gamma_1(v)$, let $Y_i$ be the number of edges from $i$ to $\Gamma_2(u) \setminus \Gamma_1(v)$, that is
    \[Y_i = \sum_{w \in \Gamma_2(u) \setminus A} X_{i,w} + \sum_{w \in (\Gamma_2(u) \setminus \Gamma_1(v)) \cap A} X_{i,w}.\]

    The second term consists of (indicators for the) edges adjacent to $u$, $v$, $x$ or $y$ and edges between the neighbourhoods of those vertices. In particular, the second term is already known (as these edges have been revealed) and we denote it by $\eps_i$. The assumptions we have made imply that $\eps_i \leq \eps$ where we have $\eps = 3\left(n^2p^3\right)^{1/4} + 4$ if $n^{-11/20} \leq p \leq n^{-16/35}$ and  $\eps = 157$ if $p \leq n^{-11/20}$.
    Provided that $i \not \in \{u,v, x, y\} \cup M$, we have not revealed any of the indicator variables in the first sum, and $Y_i - \eps_i$ is a binomial random variable with $\lambda + O((np)^{5/4})$ trials and success probability $p$.

    Similarly, for $j \in \Gamma_1(y)$, let $Y'_j$ be the number of edges from $j$ to $\Gamma_2(x) \setminus \Gamma_1(y)$, that is
    \[Y_j' = \sum_{w \in \Gamma_2(x) \setminus A} X_{j,w} + \sum_{w \in (\Gamma_2(x) \setminus \Gamma_1(y)) \cap A} X_{j,w},\]
    and let $\eps_j' = \sum_{w \in (\Gamma_2(x) \setminus \Gamma_1(y)) \cap A} X_{j,w}$. Define $B_1$ and $B_2$ by $ B_1 =  \Gamma_1(v) \setminus \left( M \cup \{u, v, x, y\}\right)$ and $B_2 = \Gamma_1(y) \setminus \left( M \cup \{u, v, x, y\} \right)$, so that the random variables \[\{Y_i - \eps_i : i \in B_1\}  \cup \{Y'_j - \eps'_j : j \in B_2 \}\]
    are independent binomial random variables, each with success probability $p$. Indeed, if $Y_{i_1} - \eps_{i_1}$ and $Y_{i_2} - \eps_{i_2}$ ($i_1 \neq i_2$) are not independent, then there must be $w_1, w_2 \in \Gamma_2(u) \setminus A$ such that $\{i_1, w_1\} = \{i_2, w_2\}$. Since $i_1 \neq i_2$, we would have $i_1 = w_2 \in \Gamma_2(u) \setminus A$, but $i_1 \in A$. If there are $i \in B_1$ and $j \in B_2$ such that $Y_i - \eps_i$ and $Y_j' - \eps_j$ are not independent, there must be $w_1 \in \Gamma_2(u) \setminus A$ and $w_2 \in \Gamma_2(x) \setminus A$ such that $\{i, w_1\} = \{j, w_2\}$. Since $i \not \in M$ and $i \in  \Gamma_1(v)$, we cannot have $i \in \Gamma_1(y)$ and so $i \neq j$. This means $i = w_2$, but then $w_2 \in \Gamma_1(v) \subseteq A$, a contradiction.

    If $C_{uv}$ is isomorphic to $C_{xy}$ with $u$ mapping to $x$, then the multisets $\{Y_i: i \in \Gamma_1(v)\}$ and $\{Y'_j : j \in \Gamma_1(y)\}$ must be equal. Equivalently, the number of $Y_i$ and $Y_j'$ equal to $k$ must be equal for every choice of $k$. The $Y_i$ with $i \not \in B_1$ are potentially problematic, but there are at most $m+4$ of them and so we ignore them and consider the multiset $\{Y_i : i \in B_1\}$ which is ``close" to the multiset $\{Y_i : i \in \Gamma_1(v)\}$. Likewise we can consider the multiset $\{Y_j' : j \in B_2\}$ which is ``close" to the multiset $\{Y_j' : j \in \Gamma_1(y)\}$. Since we have deleted at most $m + 4$ elements from each multiset, the number of $Y_i$ and $Y_j'$ equal to $k$ in the resulting multisets may differ by at most $m+4$.

    Let $Z_k$ be the number of the $Y_i$, where $i \in B_1$, which are equal to $k$ and note that $Z_k$ is the sum of $|B_1|$ independent Bernoulli random variables (with potentially different probabilities due to different $\eps_i$).
    Similarly, let $Z_k'$ be the number of the $Y_j'$, with $j \in B_2$ which are equal to $k$.

    Let $\mu = |\Gamma_2(u) \setminus A| p$ and  $\mu' = |\Gamma_2(x) \setminus A| p$, so that $\expec{Y_i - \eps_i} = \mu$ and $\expec{Y_j' - \eps_j'} = \mu'$. Since $|A| = O(np)$ and $\Gamma_2(u)$ and $\Gamma_2(x)$ are both $\lambda + O((np)^{5/4})$, both $\mu$ and $\mu'$ are $p\lambda  + O(n^{5/4}p^{9/4})$. 
    Without loss of generality let us assume that $\mu' \geq \mu$, and define $k_i$ by $ k_i= \ceil{\mu'} + \eps + i$. Let $\ell$ be a quantity to be determined. We will reveal the values of $Z_{k_i}$ for $i \in [\ell]$ and call these our \emph{target values}. If there is an isomorphism mapping $C_{uv}$ to $C_{xy}$ which sends $u$ to $x$, it must be the case that $|Z_{k_i} - Z_{k_i}'| \leq m + 4$ for all $i \in [\ell]$, and we will iteratively bound the probability that $|Z_{k_i} - Z_{k_i}'| \leq m+4$, conditional on the event that such a bound held for the values $k_1, \dots, k_{i-1}$. If this event does not occur, then $C_{uv}$ and $C_{xy}$ are not isomorphic and we are done. If the event does occur, we reveal the vertices in $B_2$ which have $k_i$ edges to $\Gamma_2(x) \setminus \Gamma_1(y)$ and carry on.

    We now prove a series of claims which we will use to ensure that the probability that $|Z_{k_i} - Z_{k_i}'| \leq m + 4$ is small for every $i$. We start by showing that knowing that $Y_j'$ has not already been revealed only changes the probability that it is revealed in the next step by a constant factor. We will then show that the probability that $Y_j'$ takes a particular value $k_i$ is small, for which we use two different approximations depending on the value of $p$.

    \begin{claim}
        \label{claim:still-many}
        For any $\ell > 0$,
        \[\prob{Z_{k_1} + \dotsb Z_{k_\ell} \leq 3|B_2|/4} =  1 - o(n^{-4}).\]
    \end{claim}
    \begin{proof}
        We first bound the probability that a given $Y_i$ is in $\{k_1, \dots, k_\ell\}$, or equivalently, that $Y_i - \eps_i \in \{k_1 - \eps_i, \dots, k_\ell - \eps_i\}$. Since $k_1 - \eps_i > \ceil{\mu}$, this is clearly bounded above by the probability that $Y_i - \eps_i > \ceil{\mu}$. The random variable $Y_i - \eps_i$ follows a binomial distribution and hence the median is $\floor{\mu}$ or $\ceil{\mu}$.  This means \[\prob{Y_i \in \{k_1, \dots, k_\ell\}} \leq \frac{1}{2}.\]
        In particular, the random variable $Z_{k_1} + \dotsb + Z_{k_\ell}$ is dominated by a binomial random variable with $|B_1| = \Theta(np)$ trials and success probability $1/2$. Using \Cref{lem:chernoff}, the probability that such a random variable exceeds $2|B_1|/3$ is at most $\exp(-|B_1|/6^3) = o(n^{-4})$. The result is now immediate since $|B_2| = (1 + o(1)) |B_1|$.
    \end{proof}

    \begin{claim}
        \label{claim:still-small-prob}
        For all $i \in [\ell]$,
        \[ \prob{Y_j' = k_i}  \leq \prob{Y_j' = k_i | Y_j' \not \in \{k_1, \dots, k_{i-1}\}} \leq 2\prob{Y_j' = k_i}.\]
    \end{claim}
    \begin{proof}
        The claim follows immediately from  $\prob{Y_j' \in \{k_1, \dots, k_\ell\}} \leq 1/2$ and
        \[\prob{Y_j' = k_i| Y_j' \not \in \{k_1, \dots, k_{i-1}\}} = \frac{\prob{Y_j' = k_i}}{1 - \prob{Y_j' \in \{k_1, \dots, k_{i-1}\}}}.\]
    \end{proof}

    We now assume that $Z_{k_1} + \dotsb + Z_{k_\ell} \leq 3|B_2|/4$. Our goal is to apply \Cref{thm:sup} for which we need to bound the probability that $Y_j' = k_i$ given that $Y_j' \not \in \{k_1, \dots, k_{i-1}\}$. We use different approaches for different values of $p$, and we now split the proof into two parts.

    \begin{claim}
        \label{claim:local-limit}
        Suppose $p = \omega(n^{-2/3})$ and $p \leq n^{-16/35}$. There exist constants $\alpha,\beta > 0$ such that, for all $j \in B_2$ and $i \in [\sqrt{\mu'}]$, we have
        \begin{align*}
            \frac{\alpha}{\sqrt{\mu'}} \leq & \prob{Y_j' = k_i} \leq \frac{\beta}{\sqrt{\mu'}}.
        \end{align*}
    \end{claim}
    \begin{proof}
        Note that $\prob{Y_j' = k_i} = \prob{Y_j' - \eps_j' = k_i - \eps_j'}$ and that $Y_j' - \eps_j'$ is a binomial random variable whose variance tends to infinity. By \Cref{thm:DML} it is enough to show that there is a constant $M$ such that $|k_i - \eps'_j - \mu' | \leq M \sqrt{\mu'}$ for all $j \in B_2$ and $k_i$. We have that
        \begin{align*}
            \left|k_i - \eps'_j - \mu' \right| & \leq |\ceil{\mu'} - \mu'| + |\eps_j'| + i \\
            & \leq  1 + \eps + \sqrt{\mu'},
        \end{align*}
        so we only need to show that $1 + \eps = O(\sqrt{\mu'})$. 
        
        As seen in the proof of \Cref{claim:neighbourhoods}, we have $\lambda \geq \tfrac{n}{2} \min \{ 1 - e^{-1}, np^2/4\}$ for large enough $n$.
        In particular, there are constants $a$ and $b$ such that $\sqrt{\mu'} \geq \min\{a\sqrt{np}, b \sqrt{n^2p^3}\}$ for large enough $n$. 
        This implies that $\sqrt{\mu'} = \omega(1)$, and it is easy to check that $(n^{2}p^{3})^{1/4} = O(\sqrt{\mu'})$ as well.
    \end{proof}

    Suppose we are at stage $i$, and so we have already revealed the vertices with degrees $k_1, \dots, k_{i-1}$ and are interested in the event that $|Z_{k_i} - Z_{k_i}'| \leq m + 4$. Since we have already revealed $Z_{k_i}$, it suffices to bound the probability that $Z_{k_i}'$ takes one of the $2m + 9$ most likely values. The random variable $Z_{k_i}'$ is the sum of independent Bernoulli random variables, and we may apply \Cref{thm:sup}. By \Cref{claim:still-many} there are at least $|B_2|/4$ trials and by \Cref{claim:still-small-prob} the success probability of each trial is at least $\alpha/\sqrt{\mu'}$ and at most $2\beta/\sqrt{\mu'}$. Since $\mu' \to \infty$ as $n \to \infty$, we may assume $2\beta/\sqrt{\mu'} < 1/2$. In particular, each unrevealed $j \in B_2$ is equal to $k_i$ with probability less than $1/2$. Applying \Cref{thm:sup} we have

    \[\sup_x \bP\left(Z_{k_i}' = x\right) \leq C\left({\frac{\alpha|B_2|}{4 \sqrt{\mu'}}}\right)^{-1/2} = O(p^{1/4}),\]
    and
    \[\prob{|Z_{k_i} - Z_{k_i}'| \leq m+4} = O\left(m p^{1/4}\right).\]
    Since $p \leq n^{-16/35}$ and $m \leq 12n^{1/9}$, we have $mp^{1/4} = O(n^{-1/315})$. The only condition on $\ell$ in this argument comes from the application of \Cref{claim:local-limit} where we required that $\ell \leq \sqrt{\mu'}$. Since $\sqrt{\mu'} = \omega(1)$, we may take $\ell > 1260$ to be a constant, in which case the probability that all $\ell$ steps succeed is $O(n^{-\ell/315}) = o(n^{-4}$) as required.

    \bigskip
    We now consider the case where $n^{-2/3 - \delta} \leq p \leq n^{-2/3} \log \log n$. Instead of applying a local limit theorem as in \Cref{claim:local-limit}, we approximate $Y_j' - \eps_j'$ by a Poisson random variable and use this to bound the probability that $Y_j' - \eps_j'$ equals $k_i$.

    \begin{claim}
        \label{claim:le-cam}
        Suppose $n^{-2/3 - \delta} \leq p \leq n^{-2/3} \log \log n$. Then, for all $i > 0$, we have
        \[ \frac{(\mu')^{k_i - \eps} \exp(-\mu')}{(k_i - \eps)!} + O\left(n^2p^4\right) \leq \prob{Y_j' = k_i} \leq  1/5 + O\left(n^2p^4\right).\]
    \end{claim}
    \begin{proof}
        By Le Cam's Theorem (\Cref{thm:le-cam}), the total variation distance between $Y_j' - \eps_j'$ and a Poisson random variable with mean $\mu'$ is at most $2p\mu' = O(n^2 p^4)$. Hence,
        \[\prob{Y_j'  = k_i } = \prob{Y_j' - \eps_j' = k_i - \eps_j'} = \frac{(\mu')^{k_i - \eps_j'}  \exp(-\mu')}{(k_i - \eps_j')!} + O\left(n^2p^4\right).\]
        The probability mass function of a Poisson distribution is decreasing above its mean, and so the right hand side is a decreasing function of $k_i - \eps_j'$. The lower bound now follows since $\eps_j' \leq \eps$. For the upper bound, note that $k_i - \eps_j' \geq \ceil{\mu'} + 1$, and it suffices to bound
        \[\frac{t^{\ceil{t + 1}} \exp(-t)}{\ceil{t + 1}!}\]
        over all values of $t > 0$.
        This is bounded above by $1/5$.
    \end{proof}

    The random variable $Z_{k_i}$ is the sum of at least $|B_2|/4$ independent Bernoulli random variables, each with probability at least ${(\mu')^{k_i - \eps} \exp(-\mu')}/{(k_i - \eps)!} + O\left(n^2p^4\right)$ and at most $2/5 + O\left(n^2p^4\right)$. Hence, by \Cref{thm:sup},
    \begin{equation}\label{eqn:zki}
        \sup_t \bP\left(Z_{k_i}' = t\right) \leq C\left(\frac{|B_2|}{4} \cdot \frac{(\mu')^{k_i - \eps} \exp(-\mu')}{(k_i - \eps)!} + O\left(n^3p^5\right)\right)^{-1/2}.
    \end{equation}

    Note that $t^t\exp(-t)$ is bounded below by $1/e$ and that $\mu' = p \lambda + O(n^{5/4}p^{9/4})$. Since $\lambda \leq 2n^2p^2$, we may assume $\mu' \leq 3 (\log \log n)^3$ for large enough $n$. We also have that $\mu' \geq \gamma n^2 p^3 \geq \gamma n^{-3\delta}$ for some small $\gamma > 0$ and large enough $n$. Hence, for large enough $n$,
    \begin{align*}
        \frac{|B_2|}{4} \cdot \frac{(\mu')^{k_i - \eps} \exp(-\mu')}{(k_i - \eps)!} & = \frac{|B_2|}{4} \cdot (\mu')^{\mu'}\exp(-\mu') \cdot \frac{(\mu')^{i + \ceil{\mu'} - \mu'}}{(\ceil{\mu'} + i)!} \\
             & \geq \frac{|B_2|}{4e}  \cdot \frac{ \gamma^{i+1} n^{-3\delta (i + 1)}}{\left( 3 (\log \log n)^3 + \ell + 1\right)!}.            \\
    \end{align*}
    For any fixed $\ell$ and $\delta$, the quantity $( 3 (\log \log n)^3 + \ell + 1)!$ is less than $n^{3\delta}$ for large $n$. We also have that $|B_2| \geq np/2 - (m + 4) \geq n^{1/3 - \delta}/3$ for large $n$.  Hence,
    \[ \frac{|B_2|}{4} \cdot \frac{(\mu')^{k_i - \eps} \exp(-\mu')}{(k_i - \eps)!} \geq \frac{\gamma^{i+1} n^{1/3 - \delta - 3\delta(i + 2)}}{12e}.\]
    Substituting this bound into (\ref{eqn:zki}) gives
    \begin{align*}
        \sup_t \bP\left(Z_{k_i}' = t\right) \leq C \left(\frac{ \gamma^{i+1} n^{1/3 - \delta - 3\delta (i + 2)}}{12e} + O(n^3p^5)\right)^{-1/2}.
    \end{align*}
    Hence, the probability that all $\ell$ steps complete is at most
    \[\prod_{i=1}^{\ell}(2m + 9)C \left(\frac{\gamma^{i+1} n^{1/3 - \delta - 3\delta (i + 2)}}{12e} + O(n^3p^5)\right)^{-1/2} = O \left( n^{ - (\ell/6 - 7\ell \delta/2 - 3 \delta \ell (\ell + 1)/4)}\right).\]
    For any $\ell > 24$, one can choose $\delta$ sufficiently small such that \[\ell/6 - 7\ell \delta/2 - 3 \delta \ell (\ell + 1)/4 > 4\]
    which completes the proof.
\end{proof}

\section{Non-reconstructibility from 1-neighbourhoods and 2-neighbourhoods}\label{sec:12non}
In this section we prove \Cref{thm:r=2-non} and \Cref{thm:r=1}. The proofs are quite similar, but differ in the technical details. We start in \Cref{sec:1non} with the proof of \Cref{thm:r=1} since it is slightly simpler, and then we move on to the proof of \Cref{thm:r=2-non} in \Cref{sec:2non}.

\subsection{1-neighbourhoods}\label{sec:1non}
In this subsection we prove \Cref{thm:r=1}. When $p=O\left(\frac {\log n}n\right)$ and $p=\omega(n^{-3/2})$, we can appeal directly to \Cref{thm:firstPhaseTransition}. It is therefore sufficient to show that if $p \leq \sqrt {\frac{\log n}{25n}}$ and $p = \omega(n^{-1})$, a random graph $G\in\mathcal G(n,p)$ is not 1-reconstructible with high probability.

\begin{proof}
    Suppose that $p = \omega(n^{-1})$ and $p \leq c \sqrt {\frac{\log n}{n}}$ for some small constant $c > 0$ (which we will later take to be $1/5$).
    We will show that with high probability, there exist four vertices $u,v,x,y \in V(G)$ such that
    \begin{enumerate}
        \item the pairs $xy,uv\in E(G)$, and $xu,xv,yu,yv\notin E(G)$,
        \item all the degrees $d(u),d(v),d(x),d(y)$ are different,
        \item the degrees $d(u),d(v),d(x),d(y)$ are at most $(np)^{2/3}$ from $np$, and
        \item the neighbourhoods $\Gamma(u),\Gamma(v),\Gamma(x)$ and $\Gamma(y)$ are all pairwise disjoint.
    \end{enumerate}
    It is straightforward to see that this implies that the graph $G$ is not reconstructible from its 1-neighbourhoods. Indeed, the graphs $G$ and $G'=\left(G\setminus\{xy,uv\}\right)\cup\{xu,yv\}$ have the same collection of 1-neighbourhoods, but they are not isomorphic as there is one fewer edge between vertices of degree $d(x)$ and $d(y)$ in $G'$ than in $G$.

    It thus remains to prove that there exist four such vertices with high probability.
    Let $A=(a_1,a_2,a_3,a_4)\subseteq V(G)$ be an ordered tuple of four vertices, and let $X_A$ be the indicator of the event that the vertices of $A$ satisfy the conditions above with $a_1=u, a_2=v, a_3=x$ and $a_4=y$. Let $X=\sum_{A\subseteq V}X_A$ be the total number of such `good' tuples. Then $\expec{X}=\sum_{A\subseteq V}\E[X_A]=4!\binom n4\prob{X_{(1,2,3,4)}=1}$.
    Let $R_1,R_2,R_3$ and $R_4$ be the events that $(1,2,3,4)$ satisfies the conditions 1, 2, 3 and 4 respectively. The probability of the event $R_1$ is simply $p^2(1-p)^4$. Given that $R_1$ occurs, the degree of a vertex in $A$ is distributed like a $\Bin(n-4,p)$ random variable plus one. The degrees are independent so the probability that two of the vertices have the same degree is at most 6 times the probability that two $\Bin(n-4, p)$ random variables are equal, and this is $o(1)$ by \Cref{thm:sup}. Further, an application of \Cref{lem:chernoff} shows that $\prob{R_3^c \mid R_1}=o(1)$, and hence, $\prob{R_2 \cap R_3 \mid R_1} = 1 - o(1)$.

    We now consider $R_4$.  Given $n'$ and $a$ with $|n'-n|\le 8$ and $|a-np|\le(np)^{2/3} + 8$, the probability that four uniformly chosen sets from $[n']$ of size $a$ are pairwise disjoint is
    \begin{align}
        \frac{\binom{n'}a \binom{n'-a}a \binom {n'-2a}{a} \binom{n'-3a}a}{\binom{n'}a ^4}= (1-o(1))e^{-6a^2/n} = (1-o(1))e^{-6np^2}. \label{swiss}
    \end{align}
    The first equality follows from rewriting the left hand side as $\frac {(n')!}{(n'-4a)!}\cdot \left(\frac{(n'-a)!}{(n')!} \right)^4$ and using Stirling's approximation.
    Given $R_1$, $R_2$ and $R_3$ the probability that $R_4$ occurs can be bounded above by the probability that four uniformly chosen sets from $[n-4]$
    of size $\ceil{np - (np)^{2/3}}$ are pairwise disjoint, and bounded below by the probability that four uniformly chosen sets from $[n-4]$ of size $\floor{np + (np)^{2/3}}$ are pairwise disjoint. By \eqref{swiss} both probabilities are $(1-o(1))e^{-6np^2}.$

    Combining the above we have $\prob{X_{A}} = (1-o(1))p^2 \expb{-6np^2}$, and so
    \begin{align}
        \E[X]=(1+o(1))n^4p^2 \expb{-6np^2} = \Omega(n^{2-6c^2}). \label{expstride}
    \end{align}

    We next show that $\expec{X^2} \leq (1 + o(1)) \expec{X}^2$, so that $\Var(X) = o(\expec{X}^2)$ and Chebyshev's inequality completes the proof. Write
    \begin{align*}
        \E[X^2] & =\sum_{A_1,A_2}\E[X_{A_1}X_{A_2}]
        =\sum_{k=0}^{4}\sum_{\substack{A_1,A_2     \\|A_1\cap A_2|=k}}\prob{(X_{A_1}=1) \wedge (X_{A_2}=1)}.
    \end{align*}

    We first consider when $A_1$ and $A_2$ intersect (with $|A_1\cup A_2|=8-k$). If both $A_1$ and $A_2$ satisfy condition 1, then there are at least $4-k/2$ edges which must each be present. This happens with probability at most $p^{4-k/2}$. Hence, summing over the at most $n^{8-k}$ choices for $A_1$ and $A_2$ for each $k$ and noting that $n^2p = \omega(1)$, we have
    \begin{align*}
        \sum_{k=1}^{4}\sum_{\substack{A_1,A_2 \\|A_1\cap A_2|=k}}\prob{(X_{A_1}=1) \wedge (X_{A_2}=1)}&\leq \sum_{k=1}^{4} n^{8-k}p^{4 - k/2}
        \leq 4 n^7p^{7/2}.
    \end{align*}
    Considering \eqref{expstride}, we see that for small enough $c$ this sum is $o(\E[X]^2)$.
    Indeed, $n^7p^{7/2} = O(n^{15/2}p^4)$ while $\expec{X}^2 = \Omega(n^{8-12c^2}p^4)$, and it suffices to take $c = 1/5$.
    It therefore suffices to show that the sum over the choices of $A_1$ and $A_2$ with no intersection contributes at most $(1+o(1))\E[X]^2$.

    Now suppose that there is no intersection between $A_1$ and $A_2$. We loosen the requirements given by 1, 2, 3 and 4, by ignoring the edges between $A_1$ and $A_2$, and ignoring condition 2. Condition 1 is unchanged, and condition 4 is weaker as we allow the neighbourhoods to intersect in $A_1$ and $A_2$. We modify condition 3 so that the degree of each vertex is at most $(np)^{2/3} + 4$ away from $np$ ignoring any edges between $A_1$ and $A_2$, and note that this has a negligible difference on the probability. Let $X'_{A_1, A_2}$ be the indicator of the event that both $A_1$ and $A_2$ pass these conditions which, since we have weakened the conditions, dominates the event that $X_{A_1} = 1$ and $X_{A_2} = 1$. Repeating the calculation from before shows that $\prob{X'_{A_1, A_2} = 1} = (1+o(1))\prob{X_{(1,2,3,4)} = 1}^2$.
    It then follows that $\sum_{A_1\subseteq V}\sum_{A_2\subseteq V\setminus A_1}\prob{(X_{A_1}=1)\wedge (X_{A_2}=1)}\leq \left(\sum_{A\subseteq V}(1+o(1))\prob{X_{A}=1}\right)^2=(1+o(1))\E[X]^2$, as required.
\end{proof}

\subsection{2-neighbourhoods}\label{sec:2non}

In this subsection we prove \Cref{thm:r=2-non}. When $p=O\left(\frac{\log n}n\right)$ and $p=\omega(n^{-5/4})$, we can appeal directly to \Cref{thm:firstPhaseTransition}, so it suffices to consider $p$ where $p \leq \frac{1}{3} \left(\frac {\log^{1/3}  n}{n}\right)^{3/4}$ and $p = \omega(n^{-1}\log \log n)$. We will show that for such $p$ a random graph $G\in\mathcal G(n,p)$ is not 2-reconstructible with high probability.

\begin{proof}[Proof of \Cref{thm:r=2-non}]
    Suppose that $p = \omega(n^{-1}\log \log n)$ and $p \leq c \left(\frac {\log^{1/3}  n}{n}\right)^{3/4}$ for some small constant $c > 0$ (which we will later take to be $1/3$).
    For 2 vertices $i\sim j$, define the `one-sided 2-neighbourhood' of $i$ with respect to $ij$ to be $N_2^{ij}(i)=(\Gamma_1(i)\setminus\{j\})\cup(\Gamma_2(i)\setminus \Gamma_1(j))$.
    We will show that with high probability, there exist four vertices $u,v,x,y \in V(G)$ such that
    \begin{enumerate}
        \item the pairs $xy,uv\in E(G)$, and $xu,xv,yu,yv\notin E(G)$,
        \item $d(x)=d(v)$ and $d(y)=d(u)$,
        \item the degrees $d(u),d(v),d(x),d(y)$ are at most $(np)^{2/3}$ from $np$,
        \item the sizes of the second neighbourhoods $|\Gamma_2(x)|$, $|\Gamma_2(y)|$, $|\Gamma_2(u)|$, $|\Gamma_2(v)|$ are all different,
        \item the sizes of the second neighbourhoods $|\Gamma_2(x)|$, $|\Gamma_2(y)|$, $|\Gamma_2(u)|$, $|\Gamma_2(v)|$ are at most $(n^2p^2)^{2/3}$ from $n^2p^2$,
        \item the  graphs induced by the first neighbourhoods are all empty (i.e. $G[\Gamma(x)]$, $G[\Gamma(y)]$, $G[\Gamma(u)]$, $G[\Gamma(v)]$ contain no edges), and
        \item the one-sided 2-neighbourhoods $N_2^{xy}(x),N_2^{xy}(y),N_2^{uv}(v)$, $N_2^{uv}(u)$ are disjoint.
    \end{enumerate}
    It is straightforward to see that this implies that the graph $G$ is not reconstructible from its 2-neighbourhoods. Indeed, conditions 1, 2, 6 and 7 ensure the graphs $G$ and $G'=\left(G\setminus\{xy,uv\}\right)\cup\{xu,yv\}$ have the same collection of 2-neighbourhoods, but the number of edges $ij$ where $|\Gamma_2(i)| = |\Gamma_2(x)|$ and $|\Gamma_2(j)| = |\Gamma_2(y)|$ (or the other way round) is one less in $G'$.

    It thus remains to prove that there exist four such vertices with high probability.
    Let $A=(a_1,a_2,a_3,a_4)\subseteq V(G)$, and let $X_A$ be the event that the vertices of $A$ satisfy the conditions above with $a_1=u, a_2=v, a_3=x, a_4=y$. Let $X=\sum_{A\subseteq V}X_A$ be the total number of such `good' tuples. Then $\expec{X}=\sum_{A\subseteq V}\E[X_A]=4!\binom n4\prob{X_{(1,2,3,4)}=1}$.
    For $i\in[7]$, let $R_i$ be the event that $(1,2,3,4)$ satisfies the condition $i$ above. The probability of the event $R_1$ is simply $p^2(1-p)^4$.
    Further, an application of \Cref{lem:chernoff} gives $\prob{R_3^c \mid R_1}=o(1)$.
    Given that $R_1$ occurs, the degree of a vertex
    in $A$ is distributed like a $\Bin(n - 4, p)$ random variable plus one. Given $R_1$, the degrees $d(u)$, $d(v)$, $d(x)$ and $d(y)$ are all independent so, since $(1-p)pn = \omega(1)$, an application of \Cref{thm:DML} shows that $\prob{R_2\mid R_1}=\Theta(\frac {1}{np} )$.

    Now reveal the edges between $u$, $v$, $x$ and $y$ and the degrees $d(u), d(v), d(x)$ and $d(y)$, and assume that $R_1$, $R_2$ and $R_3$ hold.
    Given $n'$ and $a'$ with $|n'-n|\le 8$ and $|a'-np|\le(np)^{2/3}$, the probability that four uniformly chosen sets from $[n']$ of size $a'$ are pairwise disjoint is
    \begin{align}
        \frac{\binom{n'}{a'} \binom{n'-a'}{a'} \binom {n'-2a'}{a'} \binom{n'-3a'}{a'}}{\binom{n'}{a'} ^4}= (1-o(1))e^{-6a'^2/n} = (1-o(1))e^{-6np^2}=1-o(1).  \label{swiss2.1}
    \end{align}
    Given that conditions $R_1$, $R_2$ and $R_3$ hold, the probability that $\Gamma(x),\Gamma(y),\Gamma(u),\Gamma(v)$ are disjoint can be bounded above by the probability that four uniformly chosen sets from $[n]$ of size $\ceil{np - (np)^{2/3}}$ are pairwise disjoint, and bounded below by the probability that four uniformly chosen sets from $[n-4]$ of size $\floor{np + (np)^{2/3}}$ are pairwise disjoint. By \eqref{swiss2.1} this is $(1-o(1))$.
    
    Assuming that the 1-neighbourhoods are disjoint (and $R_1, R_2$ and $R_3$ hold), $|\Gamma_2(x)|$ is distributed like a $\Bin(n - d(x) - d(y), 1 - (1-p)^{d(x) - 1})$ random variable plus $d(y) - 1$. Hence, by \Cref{thm:sup}, the probability that $|\Gamma_2(x)| = |\Gamma_2(y)|$ is $O(\frac {1}{np})$, and it follows that the probability of $R_4$ is $1 - o(1)$. Applying \Cref{lem:chernoff} also shows that the probability that $R_5$ holds is $1 - o(1)$.

    We are left with $R_6$ and $R_7$. For them to hold, we first consider the probability that $G[\Gamma(x)],G[\Gamma(y)],G[\Gamma(u)],G[\Gamma(v)]$ are all empty, and then the probability that the second neighbourhoods are disjoint, and also disjoint from the first neighbourhoods.
    We have already conditioned on the event that the first neighbourhoods are all disjoint and the probability that they are all empty (given that they are disjoint, and given $R_1,R_2,R_3,R_4,R_5$) is bounded from below by
    $1-4\bP(\Bin(\binom{\hat d}2,p)>0)$, where $\hat d= \floor{np+(np)^{2/3}}$. Since $\E[\Bin(\binom{\hat d}2,p)]=o(1)$ for our range of $p$, we obtain that the conditioned probability is $(1-o(1))$ by applying Markov's inequality.
    Finally, to complete $R_7$, note again that the probability that four uniformly chosen sets of size $a = n^2p^2 + O((n^2p^2)^{2/3})$ chosen from $[n']$ where $|n'-n|=O(np)$ are pairwise disjoint is
    \begin{align}
        \frac{\binom{n'}a \binom{n'-a}a \binom {n'-2a}{a} \binom{n'-3a}a}{\binom{n'}a ^4}= (1-o(1))e^{-6a^2/n} = (1-o(1))e^{-6n^3p^4}. \label{swiss2}
    \end{align}
    Given $R_1,R_2,R_3,R_4,R_5$ and that the first neighbourhoods are disjoint and empty, the probability that the one-sided second neighbourhoods are disjoint can be bounded above by the probability that four uniformly chosen sets from $[n]$ of size $\ceil{n^2p^2 - (n^2p^2)^{2/3}}$ are pairwise disjoint, and bounded below by the probability that four uniformly chosen sets from $[n']$ of size $\floor{n^2p^2 + (n^2p^2)^{2/3}}$ are pairwise disjoint, where $n'$ is given by $n' = \ceil{n - 4 - 4np - 4(np)^{2/3}}$. By \eqref{swiss2} this is $(1-o(1))e^{-6n^3p^4}.$

    Combining the above gives that $\E[X]= \Theta\left(n^3p\exp(-6n^3p^4)\right)$.

    We next show that $\expec{X^2} \leq (1 + o(1)) \expec{X}^2$, so that $\Var(X) = o(\expec{X}^2)$ and Chebyshev's inequality completes the proof. As before,
    \begin{align*}
        \E[X^2]=\sum_{k=0}^{4}\sum_{\substack{A_1,A_2 \\|A_1\cap A_2|=k}}\prob{(X_{A_1}=1) \wedge (X_{A_2}=1)},
    \end{align*}
    and we first consider when $A_1$ and $A_2$ intersect (with $|A_1\cup A_2|=8-k$). For condition 1 to be satisfied for both $A_1$ and $A_2$, there are at least $4 - k/2$ edges which must each be present and this happens with probability at most $p^{4-k/2}$. Summing over the at most $n^{8-k}$ choices for $A_1$ and $A_2$ for each $k$ we have
    \begin{align*}
        \sum_{k=1}^{4}\sum_{\substack{A_1,A_2 \\|A_1\cap A_2|=k}}\prob{(X_{A_1}=1) \wedge (X_{A_2}=1)}
        \leq 4 n^7p^{7/2} \leq 4c^{3/2}n^{47/8} \log^{3/8} n\cdot p^2.
    \end{align*}
    We have  $\expec{X}^2 = \Omega(n^{6 - 12 c^4}p^2)$, so for $c=1/3$ the sum over the $A_1$ and $A_2$ that intersect is $o(\expec{X}^2)$.
    It therefore suffices to show that the sum over the instances of $A_1$ and $A_2$ with no intersection contributes at most $(1+o(1))\E[X]^2$.

    As in the proof of \Cref{thm:r=1}, we count the disjoint pairs of tuples $(a_1, a_2, a_3, a_4)$ and $(a_1', a_2', a_3', a_4')$ which satisfy slightly weaker conditions. Again, these make a negligible difference to the calculations above, and we find that the expected number of pairs of tuples is $(1 + O(1)) \expec{X^2}$, but we omit the details.
\end{proof}

\section{Properties of random graphs}\label{sec:claims}
The aim of this section is to prove the claims from \Cref{sec:mains}, and in doing so complete the proofs of \Cref{thm:r=3} and \Cref{thm:r>3}.

We prove several lemmas concerning the uniqueness of $r$-balls.  In \Cref{sec:Unique2} we show that for appropriate values of $p$, the  $2$-balls of a random graph  $G\in \mathcal G(n,p)$ are typically unique, proving \Cref{clauniq2}. Then, in \Cref{sec:Unique3}  we show that the  $3$-balls of vertices of large degree are unique (again, for appropriate values of $p$), proving \Cref{clauniq3}. In \Cref{secswap} we consider when we can swap two edges, keeping the set of $3$-balls in the graph unchanged, proving \Cref{claswapr=3}, and thus completing the proof for non-reconstructibility from 3-neighbourhoods.

\subsection{Uniqueness of 2-balls}\label{sec:Unique2}
In this section, we prove \Cref{clauniq2} which gives a region for $p$ for which the $2$-balls of a random graph $G\in\mathcal G(n,p)$ are all distinct with high probability. We build on the argument of Mossel and Ross in \cite{mossel2017shotgun} and extend their result to smaller values of $p$.
In fact, we take a similar approach and we will also show that in $\mathcal G(n,p)$, with high probability, the multisets $(\left\{d(w)\right\}_{w\in \Gamma(v)})_{v \in [n]}$ are distinct.

For a vertex $v$, let us denote the multiset of the degrees of the neighbours of $v$ by $D(v) = \left\{d(w)\right\}_{w\in \Gamma(v)}$.

\begin{proof}[Proof of \Cref{clauniq2}]
    Suppose \[\zeta^2\frac{\log^2 n}{n(\log\log n)^3} \le p \le n^{-2/3-\eps}\] for some large $\zeta$ that we will fix later.
    We may impose any positive upper bound on $\eps$, and in particular, we will assume that  $\eps < 1/3$. We show that for each pair of vertices $x,y$, the event $D(x) = D(y)$ occurs with probability $o\left(n^{-2}\right)$. Taking a union bound over the $x,y$, shows that $\mathcal{G}(n,p)$ has unique $2$-neighbourhoods with high probability.

    Fix vertices $x,y$. We first reveal the set $A$ of vertices adjacent to at least one of $x$ and $y$ excluding $x$ and $y$ themselves, i.e. $A = \left(\Gamma(x) \cup \Gamma(y)\right)\setminus\{x,y\}$. So each vertex $u\in V\setminus \{x,y\}$ is in $A$  independently with probability $1-(1-p)^2$. Note that we do not yet reveal the set of edges between $\{x,y\}$ and $A$, just that each vertex in $A$ has at least one neighbour in $\{x,y\}$.

    Next we reveal the vertices in $A$ adjacent to both $x$ and $y$, and the edges inside $A$. That is, for each vertex in $A$ we connect it to both $x$ and $y$ with probability $p^2/(1-(1-p)^2)$, while each edge inside $A$ is present independently with probability $p$.

    We discount some low-probability events via the following claims.

    \begin{claim}\label{Asize}
        Let $R_1$ be the event $\{np/2 \le |A| \le 3np\}$. Then $\prob{R_1} = 1-o\left(n^{-2}\right)$.
    \end{claim}

    \begin{claim}\label{badsmatter}
        The following hold.
        \begin{enumerate}[label=(\roman*)]
            \item Let $R_3$ be the event $\{|\Gamma(x) \cap \Gamma(y)| \le 6\}$. Then $\prob{R_3}= 1- o\left(n^{-2}\right)$.
            \item Let $R_4$ be the event that there are at most $1/\eps$ edges inside $A$. Then $\prob{R_4 \mid R_1} = 1-o\left(n^{-2}\right)$.
        \end{enumerate}
    \end{claim}

    Note that independently each vertex in $A$ which is not adjacent to both $x$ and $y$, is connected to $x$ with probability $1/2$ and otherwise it is connected to $y$ (though we do not yet reveal the adjacencies). Next we reveal every edge which is not incident with $x$ or $y$. For all $k \in \mathbb{N}$ such that $|k-np| \le \tfrac{1}{4}\sqrt{np\log(np)}$ define $A_k$ by
    \begin{align*}
        A_k = \left\{z \in A: \left|\Gamma(z) \setminus \left(A\cup\{x,y\}\right)\right| = k\right\}.
    \end{align*}
    That is $A_k$ is the set of vertices which have $k$ neighbours in the rest of the graph. We would like to think of vertices in $A_k$ as the vertices in $A$ with degree exactly $k+1$, but this is not quite correct since there may be vertices which are connected to both $x$ and $y$ and to other vertices in $A$. We will therefore only consider a subset of the possible values for $k$, and we will make sure to choose only $k$ for which $A_k$ is definitely the vertices in $A$ of degree exactly $k + 1$. When $D(x) = D(y)$, the vertices $x$ and $y$ must have the same number of the neighbours of degree $k + 1$. If we are sure that $A_k$ is exactly the vertices in $A$ of degree $k+1$, the vertices in $A_k$ must be evenly split between being neighbours of $x$ and neighbours $y$, and this is unlikely to occur if $A_k$ is ``large".

    For each $k$, we say that $A_k$ is \emph{large} if $|A_k| \ge (np)^{1/4}$, and we say that $A_k$ is \emph{small} otherwise. We claim that most $A_k$ are large, and we will ignore the small $A_k$.

    \begin{claim}\label{negcor}
        Let $R_2$ be the event $\{\#\{\mbox{small }A_k\} \le (np)^{1/4}\}$. Then $\prob{R_2 \mid R_1} = 1-o\left(n^{-2}\right)$.
    \end{claim}

    Suppose $v \in A_k$. Then $v$ has degree at least $k + 1$, but it may be higher:  $v$ might be a neighbour of both $x$ and $y$ which would increase the degree by $1$ (over the minimum); there are also at most $1/\eps$ edges between vertices of $A$ with high probability, and they could all be incident to $v$, further increasing the degree by $1/\eps$. In particular, the degree of $v$ is $k+1$ if none of these ``bad" events occur, but could be as high as $k + 2+1/\eps$.
    This motivates the following definition of a good $A_k$.

    We say that a large $A_k$ is \emph{good} if for all $s$ such that $|s-k|\le 2/\eps$ the following hold.
    \begin{enumerate}
        \item Each $z \in A_s$ is connected to exactly one of $x$ and $y$.
        \item Each $z \in A_s$ has no neighbours in $A$, i.e. $\Gamma(z) \cap A = \emptyset$.
    \end{enumerate}
    We otherwise say that $A_k$ is \emph{bad}. We wish to show that there are many good $A_k$.

    Suppose that $R_i$ holds for $i=1,\ldots,4$. We claim we have few bad $A_k$. Indeed, we have at most $(np)^{1/4}$ small $A_k$. Each vertex in $\Gamma(x) \cap \Gamma(y)$ causes at most $5/\eps$ sets $A_k$ to fail condition (1), so altogether the (at most 6) vertices in $\Gamma(x) \cap \Gamma(y)$ cause at most $30/\eps$ bad $A_k$. Similarly each edge inside $A$ causes at most $10/\eps$ (doubled for each end of the edge) $A_k$ to fail condition (2), and these edges cause at most $10/\eps^2$ bad $A_k$. Altogether we have $O\left((np)^{1/4}\right)$ bad $A_k$, and so we have at least $\tfrac{1}{3}\sqrt{np\log(np)}$ good $A_k$ for sufficiently large $n$.

    Recall that when $D(x) = D(y)$, for each good $A_k$ we must have $|A_k \cap \Gamma(y)| = |A_k \cap \Gamma(x)|$. Each vertex in a good $A_k$ is adjacent to $x$ with probability $1/2$ and otherwise adjacent to $y$, and so, independently for each good $A_k$, the quantity $|A_k \cap \Gamma(x)|$ is distributed like a $\Bin(|A_k|,1/2)$ random variable. For every $m \geq 1$, the probability that a $\Bin(m, 1/2)$ random variable takes the value $m/2$, is at most $1/\sqrt{m}$. Hence, for every good $A_k$, the probability exactly half of the vertices in $A_k$ are connected to $x$ (and half to $y$) is at most $(np)^{-1/8}$. Assuming that the events $R_1$, $R_2$, $R_3$ and $R_4$ all happen, the number of good $A_k$ is at least $\tfrac{1}{3}\sqrt{np\log(np)}$. This means that
    \[
        \mathbb{P}\left(D(x) = D(y) | R_1, \dots, R_4\right) \leq \left(\frac{1}{(np)^{1/8}}\right)^{\tfrac{1}{3}\sqrt{np\log(np)}}\\
        = \exp(-\tfrac{1}{24} \sqrt{np} \log^{3/2} (np)).\]
    Since $p \ge \zeta^2\tfrac{\log^2 n}{n(\log\log n)^3}$, this is at most
    $\exp( -\tfrac{1}{6} \zeta \log n)$
    for large enough $n$, and this is $o(n^{-2})$ for large enough $\zeta$. By Claims \ref{Asize}, \ref{badsmatter} and \ref{negcor} the probability that any of $R_1$ $R_2$, $R_3$ and $R_4$ do not hold is also $o(n^{-2})$, and this proves the result for $\beta = \zeta^2$.

\end{proof}

It remains to prove the claims.

\begin{proof}[Proof of \Cref{Asize}]
    First, note that $d(x)-1\le |A| \le d(x)+d(y)$, so it suffices to bound $d(x)$ and $d(y)$.
    Using \Cref{lem:chernoff} we have
    \[\bP(d(x) - 1 \leq np/2) \leq
        \exp(- (1+o(1))np/8) = o(n^{-2}),\]
    which proves the first inequality. For the second inequality, note that at least one of  $d(x)$ and $d(y)$ must be at least $3np/2$, and we can again use \Cref{lem:chernoff} to bound this as follows.
    \[\bP(d(x) + d(y) \geq 3 np) \leq 2 \bP(d(x) \geq 3np/2) \leq 2\exp(-np/10) = o(n^{-2}). \qedhere\]
\end{proof}

\begin{proof}[Proof of \Cref{badsmatter}]
    \begin{itemize}
        \item[(i)] Note that independently each $z \neq x,y$ is connected to both $x$ and $y$ with probability $p^2.$ Thus, $|\Gamma(x) \cap \Gamma(y)|$ is distributed like a $\Bin(n-2,p^2)$ random variable and
            \begin{align*}
                \prob{|\Gamma(x) \cap \Gamma(y)| \ge 6} \le e n^6p^{12} = O(n^{-2 - 12\eps}) = o(n^{-2}).
            \end{align*}
        \item[(ii)] Conditional on $R_1$, the number of edges inside $A$ is stochastically dominated by a $\Bin\left(6(np)^2,p\right)$ random variable, and using \Cref{lem:small-mean}

            \[ \prob{\Bin\left(6(np)^2,p\right) \ge 1/\eps} \leq e (6 n^{2}p^{3})^{1/\eps} \leq e (6 n^{-3\eps})^{1/\eps} = o\left(n^{-2}\right).\qedhere\]
    \end{itemize}
\end{proof}

\begin{proof}[Proof of \Cref{negcor}]
    For each $z \in A$, define $d'(z) = |\Gamma(z) \setminus \left(A \cup\{x,y\}\right)|$. Conditionally given $|A|$, the $d'(z)$ are distributed like independent $\Bin(n-(|A|+2),p)$ random variables. Hence, for $r \in \mathbb{N}$ such that $|r-np| \le \tfrac{1}{4}\sqrt{np\log(np)}$ and $m \in [np/2,3np]$,
    \Cref{thm:DML} gives
    \begin{align*}
        \mathbb{P}(d'(z) = r \mid |A| = m) & = \prob{\Bin(n-m-2,p) = r}                                                                                                                     \\
                                           & \geq (1+o(1)) \frac{1}{\sqrt{2 \pi np}} \exp\left( -\frac{ \left(\frac{1}{4} \sqrt{np \log(np)} + 4np^2\right)^2}{2(1-p) (np - 4np^2)} \right) \\
                                           & = (1 + o(1)) \frac{1}{\sqrt{2 \pi np}} \exp\left( - (1 + o(1)) \frac{\log(np)}{32}\right)                                                      \\
                                           & = \frac{1}{\sqrt{2 \pi}} (np)^{-\frac{1+o(1)}{32} - \frac{1}{2}}.
    \end{align*}
    For large enough $n$, this is certainly at least $(np)^{-5/8}$.

    Given $|A| = m \in [np/2,3np]$, each $|A_r|$ stochastically dominates a $\Bin\left(m,(np)^{-5/8}\right)$ random variable and, for any values $r_1, \dots, r_k$ such that $|r_i - np| \leq \tfrac{1}{4} \sqrt{np \log(np)}$ for all $i \in [k]$, the number of vertices in $A_{r_1} \cup \dotsb \cup A_{r_k}$ dominates a $\Bin(m, k (np)^{-5/8})$ random variable. If $A_{r_1}, \dots, A_{r_k}$ are all small, then this set contains at most $k(np)^{1/4}$ vertices and, by Lemma~\ref{lem:chernoff}, we have
    \[\mathbb{P} \left( \Bin(m, k(np)^{-5/8}) \leq k(np)^{1/4} \right) \leq \exp \left(-\frac{(1 - 2(np)^{-1/8})^2k(np)^{3/8}}{4}\right).\]
    Rather crudely, there are at most $(\sqrt{np \log (np)})^{(np)^{1/4} + 1}$ ways of choosing $\ceil{(np)^{1/4}}$ of the $A_r$, and the probability that all of the chosen $A_r$ are small is at most $\exp (-(np)^{5/8}/8)$ for large enough $n$.
    Hence, for large enough $n$, the probability that there are at least $(np)^{1/4}$ small $A_r$ is at most
    \[\left(\sqrt{np \log (np)}\right)^{(np)^{1/4} + 1} \exp \left(-(np)^{5/8}/8\right) = o(n^{-2}). \]
\end{proof}

\subsection{Uniqueness of 3-balls}\label{sec:Unique3}
We next turn to the proof of \Cref{clauniq3}. Recall that $\tfrac{\log^{2/3}n}{n} \le p\le  \frac{\log^2 n}{n}$, and we aim to show that with high probability the $3$-balls around vertices with degree at least $np/2$ are unique. This is done by considering the degree sequences of the neighbours of a vertex. That is, for a vertex $x$ we consider the collection of multisets of the form $\left\{d(w):w \in \Gamma\left(u\right)\right\}$, for each neighbour $u$ of $x$. Given two vertices $x$ and $y$, it would be nice to appeal to a level of independence and assume the degrees of vertices at distance $2$ from $x$ or $y$ are i.i.d. binomial random variables. Therefore, our first step in the proof is to restrict ourselves to parts of the $2$-balls around $x$ and $y$ which do not interact or overlap so that we may assume this independence. We then bound the probability that two multisets of i.i.d. binomial random variables are equal, and finally pull everything together and appeal to a union bound over pairs of vertices $x$ and $y$.

\begin{proof}[Proof of \Cref{clauniq3}]
    Let $G \in \mathcal{G}(n,p)$, and fix two vertices $x, y \in V(G)$. Suppose that $d = d(x) = d(y)$, and denote the neighbourhoods of $x$ and $y$ by $\{u_1,\ldots,u_{d}\}$ and $\{v_1,\ldots,v_{d}\}$ respectively.
    For a  vertex $w \in V(G)$, let $D(w)$ be the multiset of the degrees of the neighbours of $w$, that is, $D(w) = \{ d(z) : z \in \Gamma(w)\}$. Let $\mathcal{D}_x=\{D(u_i): i \in [d] \}$, and $\mathcal{D}_y=\{D(v_i) : i \in [d] \}$. Clearly, if the 3-balls around $x$ and $y$ are isomorphic, then $\mathcal{D}_x=\mathcal{D}_y$ as multisets, and we will show that the probability that this happens is $o(n^{-2})$.

    \begin{figure}
        \centering
        \begin{tikzpicture}
            \draw (-1,0) -- (1,0);
            \draw (3 - 0.08, -2*0.489898) -- (1,0) -- (3 - 0.08, 2*0.489898);
            \draw (-3 + 0.08, -2*0.489898) -- (-1,0) -- (-3 + 0.08, 2*0.489898);
    
            \draw (-1,0) -- (0,1) -- (1,0);
            \draw[fill, red] (0,1) circle (.05);
    
            \draw (-2,-0.3) -- (-2, 0.3);
            \draw[fill, red] (-2,-0.3) circle (.05);
            \draw[fill, red] (-2,0.3) circle (.05);
    
            \draw (2,0.3) -- (3, 0.3) -- (2, 0);
    
            \draw[fill, red] (2,0.3) circle (.05);
            \draw[fill, red] (2,0) circle (.05);
            \draw[fill, black] (3,0.3) circle (.05);
    
            \draw (2, -0.3) -- (3, -0.15) -- (3, -0.45);
            \draw[fill, black] (3,-0.15) circle (.05);
            \draw[fill, black] (3,-0.45) circle (.05);
            \draw[fill, red] (2,-0.3) circle (.05);
    
            \draw[fill] (1,0) circle (.05);
            \node[below] at (1,0) {$y$};
            \draw[fill] (-1,0) circle (.05);
            \node[below] at (-1,0) {$x$};
            \draw  (2,0) ellipse (0.2 and 0.5);
            \draw  (3,0) ellipse (0.4 and 1);
            \draw (-2,0) ellipse (0.2 and 0.5);
            \draw (-3,0) ellipse (0.4 and 1);
    
        \end{tikzpicture}
        \caption{An edge $xy$ with examples of vertices failing the conditions 2, 3 and 4 shown in red.}
        \label{fig:bad-vertices}
    \end{figure}
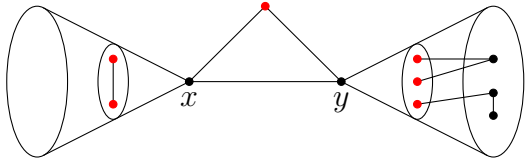

    We say that a vertex $v \in \Gamma(x) \cup \Gamma(y)$ is \emph{bad} if any of the following hold, and otherwise we say that it is \emph{good}. See \Cref{fig:bad-vertices} for examples of vertices which fail conditions 2, 3 and 4.
    \begin{enumerate}
        \item $v \in \{x,y\}$,
        \item $v$ is adjacent to both $x$ and $y$,
        \item $v$ is adjacent to a vertex in $(\Gamma(x) \cup \Gamma(y)) \setminus \{x,y\}$,
        \item there is a neighbour of $v$ adjacent to a vertex at distance at most 2 from $x$ or $y$ which is not $v$, and
        \item the degree of $v$ is less than $np/2$.
    \end{enumerate}

    We first claim that, with probability $1 - o(n^{-2})$, there are at most $2\log^{1/2} n$ bad vertices. Note that we will only be interested in applying this when $d \geq np/2 \geq \log^{2/3} (n)/2$, and so the proportion of bad vertices will tend to 0.

    \begin{claim}\label{claim:3-good}
        For any two vertices $x$ and $y$, the number of bad vertices in $\Gamma(x) \cup \Gamma(y)$ is at most $2 \log^{1/2} n$ with probability $1 - o(n^{-2})$.
    \end{claim}

    We now reveal the $2$-balls around $x$ and $y$. If $d(x) \neq d(y)$, then the $3$-balls are not isomorphic and we are done, and if $d = d(x) = d(y)$ is less than $np/2$, there is nothing to prove. From the $2$-balls, we can also check which of the vertices in $\Gamma(x) \cup \Gamma(y)$ are bad, and we assume that there are at most $2\log^{1/2} n$ of them. The degree of a vertex is dominated by a $\Bin(n, \log^2  n/n)$ random variable so, by \Cref{lem:chernoff}, we may also assume that $d \leq 2 \log^2 n$ and that the union of the $2$-balls around $x$ and $y$ contains at most $9\log^4 n$ vertices.
    If $w$ is a neighbour of a good vertex (and not $x$ or $y$), then $d(w) - 1$ is a binomial random variable, and moreover, the degrees for such vertices are i.i.d.\ random variables. Hence, if $u_i$ is a good vertex, the set $D(u_i)$ consists of $d(u_i)$ i.i.d.\ binomial random variables with at least $n - 9 \log^4  n$ trials and success probability $p$. The following claim shows that the probability that $D(u_i) = D(v_j)$ is small (for $i$ and $j$ such that $u_i$ and $v_j$ are both good).

    \begin{claim}\label{claim:equal-multisets}
        Let $A_1, \dots, A_d$ and $B_1, \dots, B_d$ be i.i.d.\ binomial random variables with $n - \sqrt{n} \leq N  \leq n$ trials and success probability $p \leq 1/2$, and suppose that $d \geq np/2$. If $np \to \infty$, then the probability that $A = \{A_1, \dots, A_d\}$ and $B = \{B_1, \dots, B_d\}$ are equal as multisets is at most \(\expb{-\Omega(\sqrt{np}\log(np))}.\)
    \end{claim}

    If $\mathcal{D}_x$ and $\mathcal{D}_y$ are equal as multisets, then there is a permutation $\sigma$ such that $D(u_i) = D(v_{\sigma(i)})$ for all $i \in [d]$. We show that, given that there are not too many bad vertices, the probability this holds for any particular choice of $\sigma$ is $o(1/(n^2 d!))$, and a union bound over the possible permutations and then the choices for $x$ and $y$ completes the proof. Let $\pi$ be a permutation of $[d]$, and consider each $i = 1, \dots, d$ in turn. If at least one of $u_i$ or $v_{\pi(i)}$ is bad, we continue onto the next $i$. If neither $u_i$ nor $v_{\pi(i)}$ is bad, then \Cref{claim:equal-multisets} shows that the probability that $D(u_i) = D(v_{\pi(i)})$ is at most $\expb{-\Omega(\sqrt{np}\log(np))}$. Since we have assumed that there are at most $2\log^{1/2}n$ vertices which are bad, we skip at most $4 \log^{1/2}n$ choices for $i$. Hence, the probability that $D(u_i) = D(v_{\pi(i)})$ for all $i \in [d]$ is at most $\expb{-\Omega(d\sqrt{np}\log(np))}$.
    By the union bound, the probability that $\mathcal{D}_x$ and $\mathcal{D}_y$ are equal is at most
    \[\prob{\mathcal{D}_x = \mathcal{D}_y}  = o(n^{-2})  +  \expb{-\Omega(d\sqrt{np}\log(np)) + d \log d}.\]
    Since $np \to \infty$ and $d \leq 2 (np)^{3}$, this is
    \(o(n^{-2}) + \expb{-\Omega(d\sqrt{np}\log(np))} = o(n^{-2})\).

    Finally, taking a union bound over the vertices $x$ and $y$ completes the proof.
\end{proof}

We now prove the two claims made in the proof above.

\begin{proof}[Proof of \Cref{claim:3-good}]
    We will bound the number of vertices that fail each of the conditions in the definition of being good. Clearly at most two vertices fail the first condition. The number of vertices which fail the second condition is given by a $\Bin(n - 2, p^2)$ random variable, which is dominated by a $\Bin(n, \log^4 (n)/n^2)$ random variable. Hence, using \Cref{lem:small-mean}, the probability there are at least three vertices which fail the second condition is at most $e\log^{12} (n)/n^3 = o(n^{-2})$.

    Consider the vertices in $\Gamma(x) \cup \Gamma(y)$ which are not one of $x$ or $y$. Using \Cref{lem:chernoff}, we may assume that there are at most $4 \log^2 n$ of them. At this point, we have only revealed the edges incident to $x$ and $y$, and so each edge $uv$ between two of these vertices is present independently with probability $p$. Hence, the number of such edges is at most $3$ with probability $o(n^{-2})$, and at most six vertices fail the third condition.

    We split the fourth condition into two parts. First, we consider the number of $v$ that fail due to one of their neighbours being adjacent to another vertex in $\Gamma(x) \cup \Gamma(y)$. A vertex $z \not \in \{x,y\} \cup \Gamma(x) \cup \Gamma(y)$ has a binomial number of neighbours in $\Gamma(x) \cup \Gamma(y)$ with at most $4 \log^2 n$ trials and success probability at most $\log^2 (n)/n$. Hence, the probability that $z$ has at least $4$ such neighbours is $o(n^{-3})$, and with probability $ 1- o(n^{-2})$, there is no choice for $z$ with at least $4$ neighbours. The probability that a vertex $z \not \in \{x,y\} \cup \Gamma(x) \cup \Gamma(y)$ has at least two neighbours in $\Gamma(x) \cup \Gamma(y)$ is at most $e(4 p\log^2 n)^2$, and so the number of such $z$ is at dominated by a $\Bin(n, 16 e \log^8 (n)/n^2)$ random variable. In particular, with probability $1 - o(n^{-2})$, there are at most $2$ vertices adjacent to least 2 vertices in $\Gamma(x) \cup \Gamma(y)$, and they are adjacent to at most 3 vertices. Hence, at most six vertices fail the first part of the fourth condition.

    Let $W$ be the set of $v \in \Gamma(x) \cup \Gamma(y)$ which have not already failed. We can reveal the set $W$ by checking the edges from $x$ and $y$ and from $\Gamma(x)$ and $\Gamma(y)$, and note that we may assume that $|\Gamma(W) \setminus \{x, y\}| \leq 4 \log^2 n$ as this happens with probability $1 - o(n^{-2})$. Hence, the number of edges between vertices in  $\Gamma(W) \setminus \{x,y\}$ is dominated by a $\Bin(16 \log^4 n, \log^2 (n)/n)$ random variable. In particular, there are at most two edges with probability $ 1 - o(n^{-2})$. Each of these can rule out at most two $v \in W$. Hence, at most a further four $v$ fail here.

    Let $W' = (\Gamma(x) \cup \Gamma(y)) \setminus \{x,y\}$. We now consider the number of vertices in $W'$ which have degree less than $np/2$. Such a vertex must have less than $np/2$ neighbours in $V \setminus (\{x,y\} \cup \Gamma(x) \cup \Gamma(y))$. We assume that we have revealed the edges from $x$ and $y$ and the edges between vertices in $\Gamma(x) \cup \Gamma(y)$, but no other edges. We may assume that there are at most $4 \log^2 n$ vertices in $\Gamma(x) \cup \Gamma(y)$. For a given vertex in $v \in W'$, the number of neighbours in $V \setminus (\{x,y\} \cup \Gamma(x) \cup \Gamma(y))$ dominates a binomial random variable with $n - 4 \log^2 n - 2$ trials and success probability $p$. Hence, the probability that it is less than $np/2$ is at most
    \(\expb{-np/16}\)
    for large enough $n$. Since each vertex $v \in W'$ satisfies this independently, the number of vertices in $W$ which have degree less than $np/2$ is dominated by a binomial random variable with $4\log^2 n$ trials and success probability $\exp(-np/16)$. Hence, the probability there are more than $\log^{1/2} n$ such vertices is at most
    \[e \left( 4 \log^2 (n) \exp\left(-\tfrac{\log^{2/3} n}{16}\right)\right)^{\log^{1/2} n} = \exp\left( \log^{1/2} n \left(\Theta(\log \log n) - \Theta(\log^{2/3} n)\right)\right),\]
    which is $o(n^{-2})$. Hence, with probability $o(n^{-2})$, the number of vertices which are bad is at most $2 + 2 + 6 + 4 + \log^{1/2} n$, as required.
\end{proof}

We now prove \Cref{claim:equal-multisets}. The general strategy here is similar to the approach used in \Cref{lem:2-neighbourhoods} when we also wanted to show that the probability that two multisets were equal was small: we count the number of $A_i$ and $B_i$ which are equal to $k$ for $\sqrt{np}$ values of $k$ close to the mean. The probability that these two quantities are equal is $O(1/\sqrt{dnp})$, and this holds even after we have revealed this for $\sqrt{np}$ choices of $k$. However, while the general strategy is similar, this time it is much simpler as the $A_i$ and $B_i$ are i.i.d. binomial random variables.

\begin{proof}[Proof of \Cref{claim:equal-multisets}]
    Let $Z_k$ be the number of $A_1, \dots, A_{d}$ which are equal to $k$ and similarly define $Z_{k}'$ to be the number of $B_1, \dots, B_{d}$ equal to $k$. Let $\ell = \ceil{\sqrt{np}} - 2$, and define $k_i = \ceil{np} + i$ for $i \in [\ell]$. By \Cref{lem:median}, we have
    \[\bP\left(B_1  \in \{k_1, \dots, k_\ell\}\right) \leq \bP \left( B_1 > \ceil{Np}\right) \leq 1/2.\]
    Hence,
    \begin{align*}
        \bP \left(Z'_{k_1} + \dotsb + Z'_{k_\ell} \geq 3d/4\right) \leq \prob{\Bin(d, 1/2) \geq 3d/4}
        \leq \exp(-d/20).
    \end{align*}
    Suppose that $Z'_{k_1} + \dotsb + Z'_{k_\ell} \leq 3d/4$ and reveal the values $Z'_{k_i}$, which we call our \emph{target values}. We will iteratively reveal the $A_j$ which are equal to $k_i$, and check if there are $Z_{k_i}'$ of them. Suppose we are about to reveal the $A_j$ equal to $k_i$, so we have already revealed the values $Z_{k_1}, \dots, Z_{k_{i-1}}$ and they are equal to $Z_{k_1}', \dots, Z_{k_{i-1}}'$. We will show that the probability that $Z_{k_i}$ is equal to $Z'_{k_i}$ is $O(1/\sqrt{np})$. Suppose that $A_j$ has not been revealed, so we know that $A_j$ is not equal to $k_1, \dots, k_{i-1}$. We have
    \[|k_{i} - Np| \leq |k_{i} - np| + |Np - np| \leq i + 1 + p\sqrt{n} \leq 2 \sqrt{np}\]
    for large $n$, and so by \Cref{thm:DML}, we have
    \begin{align*}
    \prob{A_1 = k_i}  &\leq (1 + o_\sigma(1)) \frac{1}{\sqrt{2 \pi N p (1-p)}},\\
    \prob{A_1 = k_i}  &\geq (1 + o_\sigma(1)) \frac{1}{\sqrt{2 \pi N p (1-p)}} \exp\left( - \frac{4 np}{2 Np (1-p)}\right),
    \end{align*}
    where the $o_\sigma(1)$ terms depend only on $\sigma$.
    Hence, for large enough $n$, there are constants $\alpha$ and $\beta$ such that 
    \[\frac{\alpha}{\sqrt{Np(1-p)}} < \prob{ A_1 = k_{i}} < \frac{\beta}{\sqrt{Np(1-p)}}.\]
    Since $\prob{A_j \in  \{k_1, \dots, k_{i-1}\}} \leq 1/2$, we have
    \[\mathbb{P}(A_j = k_i) \leq \mathbb{P}(A_j = k_{i} | A_j \not \in \{k_1, \dots, k_{i-1}\}) \leq 2\mathbb{P}(A_1 = k_i), \]
    and the probability that an unrevealed $A_j$ is equal to $k_i$ is $\Theta(1/\sqrt{Np})$.
    We have so far revealed $Z'_{k_1} + \dots Z'_{k_{i-1}}$ of the $A_j$ and there are at least $d/4$ unrevealed $A_j$, each of which independently takes the value $k_i$ with probability  $\Theta(1/\sqrt{Np})$.
    Hence, applying \Cref{thm:sup} gives
    \[ \prob{Z_{k_i} = Z_{k_i}'} \leq \sup_x \prob{Z_{k_i} = x} = O\left( \frac{1}{\sqrt{d/\sqrt{Np}}} \right) = O \left( \frac{1}{(np)^{1/4}}\right).\]
    If $A$ and $B$ are equal as multisets, then either $Z'_{k_1} + \dotsb + Z'_{k_\ell} > d/4$ or all of the steps succeed, and both of these happen with probability $\expb{-\Omega(\sqrt{np} \log(np))}$.
\end{proof}

\subsection{The set of 3-balls after swapping edges}\label{secswap}

In this section we prove \Cref{claswapr=3}, that is, we show that there is a constant $\alpha > 0$ such that if $\frac{\log^{2/3} n}{n}\leq p \leq \alpha \frac {\log^{2}  n}{n(\log\log n)^3}$, a random graph $G\in\mathcal G(n,p)$ is not 3-reconstructible with high probability. The main idea of the proof will be to show that, with high probability, there exist two edges $xy,uv$ in $G$ such that by deleting these edges and adding $xv,yu$ we obtain a graph $G'$ which is not isomorphic to $G$, but has the same collection of 3-balls.
\Cref{clauniq3} shows that we may assume the 3-balls around vertices of ``large" degree are all distinct, in which case, if $u, v, x$ and $y$ all have large degree, the graphs $G$ and $G'$ are not isomorphic.
To find the edges to swap we consider the structures $H_{uv}$ defined as follows. For an edge $uv$, let $H_{uv}$ be the subgraph $G[\Gamma_{\leq 2}(u) \cup \Gamma_{\leq 2}(v)]$ induced by the vertices at distance at most 2 from $u$ or $v$, and distinguish the edge $uv$. We will only consider the $H_{uv}$  for ``good" edges whose 5-balls are trees and where all the vertices in $H_{uv}$ have ``typical" degrees. There are many good edges but not that many isomorphism classes for the $H_{uv}$, and so, by the pigeonhole principle, there must be two edges $uv$ and $xy$ with $H_{uv} \simeq H_{xy}$. This is not quite enough to guarantee that the switch does not change the 3-balls by introducing extra edges and we will also require that the edges are far apart.

\begin{proof}[Proof of \Cref{claswapr=3}]
    Let $G \in \mathcal G(n,p)$ where $\frac{\log^{2/3} n}{n} \leq p \leq \alpha \frac{\log^2 n}{n (\log \log n)^3}$. We will show there exist  vertices $u, v, x, y$ as claimed using a pigeonhole argument over the $H_{uv}$ of good edges.  We say that an edge $uv$ is \emph{good} if $G[\Gamma_{\leq 5}(u) \cup \Gamma_{\leq 5}(v)]$ is a tree and $\left|d(z) - (n-1)p\right| < 10\sqrt{np\log(np)}$ for every $z \in \Gamma_{\leq 2}(u) \cup \Gamma_{\leq 2}(v)$. We will need the following claim which bounds the number of ``pigeonholes''.

    \begin{claim}\label{holes}
        The number of isomorphism classes for the $H_{uv}$ of the good edges is at most
        \begin{align*}
            400np\log(np)\expb{42\left((np)^{1/2}\log^{3/2}(np)\right)}
        \end{align*}
        for large enough $n$.
    \end{claim}

    Having bounded the number of pigeonholes, we now consider the number of pigeons, or the number of good edges $uv$ in $G$. The following claim will imply that there are at least $n^2p/8$ good edges with high probability.

    \begin{claim}\label{claimz}
        With probability $1-o(1)$, the graph $G$ satisfies the following:
        \begin{enumerate}[label=(\roman*)]
            \item The number of edges of $G$ contained in a cycle of length at most $12$ is at most $\log^{24} n$.
            \item The maximum degree of $G$ is at most $\log^2 n$.
            \item The number of vertices $z$ with degree $d(z)$ such that $|d(z) - (n-1)p| > 10\sqrt{np\log(np)}$ is at most $n^{-31}p^{-32}$.
            \item $G$ contains at least $n^2p/4$ edges.
            \item The $3$-balls around vertices of degree at least $np/2$ are all distinct.
        \end{enumerate}
    \end{claim}

    Let us denote the subgraph of $G$ induced by the vertices at distance at most 5 from $u$ or $v$ by $N_5(u,v)$, i.e. $N_5(u,v) = G[\Gamma_{\leq 5}(u) \cup \Gamma_{\leq 5}(v)]$.
    We note that if $N_5(u,v)$ is not a tree, then it contains a cycle of length at most $12$, so it will be enough to count the number of edges $uv$ such that $N_5(u,v)$ does not contain a cycle of length at most 12 and every $z \in V(H_{uv})$ satisfies the degree condition that $|d(z) - (n-1)p| \leq 10\sqrt{np\log(np)}$. For this we will first bound the number of edges $uv$ for which there is a vertex $z\in H_{uv}$ with $|d(z) - (n-1)p| > 10\sqrt{np\log(np)}$, and then we will bound the number of edges $uv$ for which there is an edge $e\in N_5(u,v)$ that is contained in a cycle of length 12 in $G$. The sum of these two bounds will be an upper bound on the number of bad edges.

    Assume that the graph $G$ satisfies the conditions given in \Cref{claimz}. Then the second condition implies that there are at most $\log^{2k} n$ vertices in the $k$th neighbourhood of a vertex, and hence every vertex $x$ is in at most $\log^2 (n) (\log^4 n + \log^2 n + 1) \leq 2 \log^6 n$ of the $H_{uv}$. Indeed, the number of vertices $u$ such that $x\in \Gamma_{\leq 2}(u)$ is at most $1+\log^2 n+\log^4 n$, and there are at most $\log^2 n$ possible different subgraphs $H_{uv}$ for each vertex $u$. In particular, a vertex $z$ with $|d(z) - (n-1)p| > 10\sqrt{np\log(np)}$  can be contained in at most $2 \log^6 n$ subgraphs $H_{uv}$. Thus, given the third condition above, the number of edges $uv$ such that $H_{uv}$ contains such a vertex $z$ is at most $n^{-31}p^{-32} \cdot 2 \log^6 n$. Similarly, each vertex is in at most  $2 \log^{12} n$ of the $N_5(u,v)$ so clearly each edge is in at most $2 \log^{12} n$ of the $N_5(u,v)$. Thus, given the first condition above, the number of edges $uv$ such that $N_5(u,v)$ contains an edge which is in a cycle of length at most 12 is $2 \log^{12} n \cdot \log^{24}n$.
    Hence, the number of bad edges for our range of $p$ is at most
    \[ 2\log^{12} n\cdot \log^{24} n + 2 \log^{6} n \cdot n^{-31} p^{-32} \leq 2 \log^{36} n + 2 n \log^{-14} n \leq n\]
    for large enough $n$.
    
    From the fourth condition $G$ has at least $n^2p/4 \geq n \log^{2/3} (n)/4$ edges and therefore (crudely) there are at least $n^2p/8$ good $H_{uv}$ for large enough $n$.

    We now use \Cref{holes} to finish the proof. There must be some isomorphism class of $H_{uv}$ that occurs at least
    \begin{multline*}
        \frac{n^2p}{8\cdot 400np\log(np)\exp(42(np)^{1/2}\log^{3/2}(np))}
        \geq \exp(\log n -43(np)^{1/2}\log^{3/2}(np))
    \end{multline*}
    times (for large enough $n$). That is, there is some good structure $J$ which appears as $H_{uv}$ for at least this many edges $uv$. Noting that $p \leq \alpha \frac{\log^2n}{n(\log\log n)^3}$, this is at least
    \[\exp\big((1 - 43\sqrt{8\alpha})\log n\big)\geq 4 \log^{14} n + 1,\]
    if $\alpha$ is sufficiently small (and $n$ sufficiently large). Suppose that $H_{uv} \simeq J$. There are at most $2 (\log^2 n)^6$ vertices at distance at most 6 from any vertex $w$, and there are at most $4\log^{12} n$ vertices at distance at most 6 from $u$ or $v$. Hence, there are at most $4 \log^{14} n$ edges where at least one vertex is at distance at most 6 from $u$ or $v$.
    Thus, there is a good edge $xy$ such that $H_{xy} \simeq J$ and both $x$ and $y$ are at distance at least $7$ from both $u$ and $v$.

    Fix an isomorphism from $H_{uv}$ to $H_{xy}$ and suppose without loss of generality that $u$ is mapped to $x$. Let $G' = \left(G \setminus \{uv,xy\}\right) \cup \{uy, vx\}$. We claim that $G'$ has the same collection of $3$-balls as $G$ and that $G'$ is not isomorphic to $G$.

    \begin{figure}
        \centering
        \begin{tikzpicture}
            \draw (-1, 0) -- (1, -3);
            \draw[blue] (1, 0) -- (-1, -3);
    
            \draw[blue] (2, 0.3) -- (1, 0);
            \draw[blue] (5 - 0.03, 0.3 - 3*.248747) -- (2, 0.3) -- (5 - 0.03, 0.3 + 3*.248747);
            \draw[fill, purple] (2,0.3) circle (.05);
            \node[above=0.3cm] at (2, 0.3) {$w$};
            \draw[blue]  (3,0.3) ellipse (0.1 and 0.25);
            \draw[blue]  (4,0.3) ellipse (0.2 and 0.5);
            \draw[blue]  (5,0.3) ellipse (0.3 and 0.75);
    
            \draw[dashed] (-1,0) -- (1,0);
            \draw[blue] (3 - 0.08, -2*0.489898) -- (1,0) -- (3 - 0.08, 2*0.489898);
            \draw (-3 + 0.08, -2*0.489898) -- (-1,0) -- (-3 + 0.08, 2*0.489898);
    
            \draw[fill, blue] (1,0) circle (.05);
            \node[above] at (1,0) {$v$};
            \draw[fill] (-1,0) circle (.05);
            \node[above] at (-1,0) {$u$};
            \draw[blue]  (2,0) ellipse (0.2 and 0.5);
            \draw[blue]  (3,0) ellipse (0.4 and 1);
            \draw (-2,0) ellipse (0.2 and 0.5);
            \draw (-3,0) ellipse (0.4 and 1);

            \begin{scope}[shift={(0, -3)}]
                \draw[dashed] (-1,0) -- (1,0);
                \draw (3 - 0.08, -2*0.489898) -- (1,0) -- (3 - 0.08, 2*0.489898);
                \draw (-2 + 0.04, 0.489898) -- (-3 + 0.08, 2*0.489898);
                \draw (-2 + 0.04, -0.489898) -- (-3 + 0.08, -2*0.489898);
    
                \draw[fill] (1,0) circle (.05);
                \node[below] at (1,0) {$y$};
                \draw[fill, blue] (-1,0) circle (.05);
                \node[below] at (-1,0) {$x$};
                \draw (2,0) ellipse (0.2 and 0.5);
                \draw (3,0) ellipse (0.4 and 1);
    
                \draw[blue] (-2,0) ellipse (0.2 and 0.5);
                \draw (-3,0) ellipse (0.4 and 1);
                \draw[blue] (-1,0) -- (-2 + 0.04, 0.489898);
                \draw[blue] (-1,0) -- (-2 + 0.04, -0.489898);
            \end{scope}
    
            \draw[red, dashed] (5, 0.1) -- (-2, -3.1);
            \draw[fill, red] (5,0.1) circle (.05);
            \draw[fill, red] (-2,-3.1) circle (.05);
    
        \end{tikzpicture}
        \caption{The $3$-ball around a vertex $w$ in the neighbourhood of $v$ in $G'$ is shown in blue. The assumption that $H_{uv} \simeq H_{xy}$ does not rule out the existence of the red edge, but this edge would create a path from $v$ to $x$ of length $6$ in $G$.}\label{fig:switch-3}
    \end{figure}
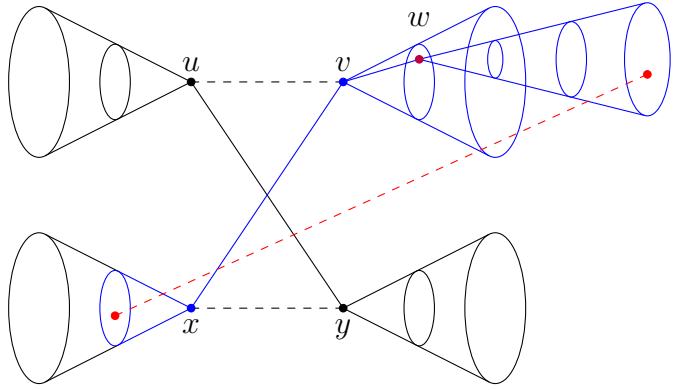

    Note that the $3$-ball of a vertex $w$ is clearly unchanged if $w$ is not in the $2$-ball of one of $u$, $v$, $x$ or $y$, so suppose it is in $\Gamma_{\leq 2}(v)$. Since $N_5(u,v)$ is a tree, the $3$-ball of $w$ in $G$ is a tree $T$. As $H_{uv} \simeq H_{xy}$ (with $u$ mapping to $x$), the $3$-ball of $u$ in $G'$ certainly contains a copy $T'$ of $T$, but this condition alone does not rule out the possibility that $w$ contains extra edges between $T' \cap T$ and $T' \setminus T$ (see \Cref{fig:switch-3} for an example).  However, any extra edge would create a cycle of length at most $7$ and it must use the edge $vx$. This means that $v$ and $x$ are at distance at most $6$ in $G$, which contradicts the choice of $xy$.

    The graphs $G$ and $G'$ cannot be isomorphic as the 3-balls around vertices of degree at least $np/2$ are unique and $G$ contains an edge between a vertex with 3-ball $N_3(u)$ and a vertex with $3$-ball $N_3(v)$ while $G'$ does not.
\end{proof}

It remains to prove our technical claims.
\begin{proof}[Proof of \Cref{holes}]
    When $uv$ is a good edge, the structure $H_{uv}$ is a tree with a distinguished edge where each vertex $z \in V(H_{uv})$ satisfies $\left|d(z) - (n-1)p\right| < 10\sqrt{np\log(np)}$. It suffices to bound the number of different options for $d(u)$, $d(v)$ and the multisets $\{d(z) : z \in \Gamma(u) \setminus v\}$ and $\{d(z) : z \in \Gamma(v) \setminus u\}$.
    The condition $\left|d(z) - (n-1)p\right| < 10\sqrt{np\log(np)}$ means that all the degrees are one of at most $N = \floor{20\sqrt{np\log(np)}} + 1$ options.
    Hence, the multiset $\{d(z) : z \in \Gamma(u) \setminus v\}$ is a multiset of $d(u) - 1$ entries spread across at most $N$ options, and so there are at most
    \begin{align*}
        \binom{d(u) + N - 2}{N - 1} & \leq (d(u) + N)^N                                                     \\
                                    & \leq \left( np + 30 \sqrt{np \log(np)}\right)^{20 \sqrt{np \log(np)} + 1} \\
                                    & \leq \exp\left(21\sqrt{np}\log^{3/2}(np)\right)
    \end{align*}
    possible multisets for large enough $n$.
    The same is true for the multiset $\{d(z) : z \in \Gamma(v) \setminus u\}$. This means there are at most
    \begin{multline*}
        \left(20\sqrt{np\log(np)} \exp\left(21(\sqrt{np}\log^{3/2}(np))\right) \right)^2\\ = 400 np\log(np)\exp\left(42\left(\sqrt{np}\log^{3/2}(np)\right)\right)
    \end{multline*}
    possible isomorphism classes for the $H_{uv}$ of a good edge, as required.
\end{proof}

\begin{proof}[Proof of \Cref{claimz}] Let $G\in \mathcal G(n,p)$. We show that each of the conditions holds with probability $1 - o(1)$, and the union bound over the five events completes the proof.
    \begin{itemize}
        \item[(i)] For each $k \in \{3,\ldots,12\}$, let $C_k$ be the number of cycles of length $k$ in $G$. Then $\E[C_k] \le n^kp^k$. For the range of $p$ that we consider, we have $np= o\left(\log^2 n\right)$ and so the expected number of edges in cycles of length at most $12$ is bounded by
            \begin{align*}
                \sum_{k=3}^{12} k\E[C_k] \le \sum_{k=3}^{12} kn^kp^k =  o\left(\log^{24} n\right).
            \end{align*}
            The claim now follows from Markov's Inequality.

        \item[(ii)] Note that the degree $d(z)$ of a vertex $z$ is distributed like a $\Bin(n-1,p)$ random variable. For large enough $n$, we have $p \leq \log^2 (n)/(2n - 2)$ and so \Cref{lem:chernoff} gives
            \begin{align*}
                \bP(d(z) \ge \log^2 n) \leq \bP\left(\Bin\big(n-1,\tfrac{\log^2 n}{2n - 2}\big) \geq \log^2 n \right) \le \exp\left(-\tfrac{1}{6}\log^2 n\right) = o(n^{-1}).
            \end{align*}
            The claim now follows from a union bound.

        \item[(iii)] Again applying \Cref{lem:chernoff} we get
            \begin{align*}
                \bP\left(|d(z) - (n-1)p| > 10\sqrt{np\log(np)}\right) \le 2\exp\left(- \tfrac{100}{3}\log(np)\right) \leq 2(np)^{-33}.
            \end{align*}
            Thus, the expected number of vertices $z$ with $|d(z) - (n-1)p| > 10\sqrt{np\log(np)}$ is bounded by $2n^{-32}p^{-33}$. We are then done by Markov's Inequality since $np \rightarrow \infty$.

        \item[(iv)] The number of edges in $G$ is distributed like a $\Bin\left(\binom{n}{2},p\right)$ random variable so the result follows from \Cref{lem:chernoff}.
        \item[(v)] This follows from \Cref{clauniq3}.
    \end{itemize}
\end{proof}

\paragraph{Acknowledgments}
The authors would like to thank the anonymous referees for their  helpful comments.

\bibliographystyle{abbrv}
\bibliography{shotgun-reconstruction}

\end{document}